\documentclass[11pt,reqno]{amsart}
\usepackage{amssymb}
\usepackage{amsmath, amssymb, amsthm}
\usepackage{mathrsfs}

\newtheorem{theorem}{Theorem}[section]
\newtheorem{definition}{Definition}[section]
\newtheorem{lemma}{Lemma}[section]
\newtheorem{remark}{Remark}[section]

\numberwithin{equation}{section}

\begin{document}
\title[Navier-Stokes-Boltzmann equations]{Existence results for  compressible  radiation hydrodynamic equations with vacuum }

\author{yachun li }
\address[Y. C. Li]{Department of Mathematics and Key Lab of Scientific and Engineering Computing (MOE), Shanghai Jiao Tong University,
Shanghai 200240, P.R.China} \email{\tt ycli@sjtu.edu.cn}

\author{Shengguo Zhu}
\address[S. G. Zhu]{Department of Mathematics, Shanghai Jiao Tong University,
Shanghai 200240, P.R.China; School of Mathematics, Georgia Tech,
Atlanta 30332, U.S.A.}
\email{\tt zhushengguo@sjtu.edu.cn}

\begin{abstract}In this paper, we consider the three-dimensional compressible isentropic  radiation hydrodynamic (RHD) equations. The existence of unique local strong solutions is firstly proved when the initial data are arbitrarily large, contain vacuum and satisfy some initial layer compatibility condition. The initial mass density does not need to be bounded away from zero and may vanish in some open set. We also prove that if the initial vacuum is not so irregular, then the initial layer compatibility condition  is necessary and sufficient to guarantee the existence of a unique strong solution. Finally, we establish a blow-up criterion for the strong solution that we obtained. The similar results also hold for the barotropic flow with general pressure law $p_m=p_m(\rho)\in C^1(\mathbb{\overline{R}}^+)$.
\end{abstract}

\date{Dec. 05, 2013}

\maketitle

\section{Introduction}

The system of radiation hydrodynamic equations appears  in high-temperature plasma physics \cite{sjx} and  in various astrophysical contexts \cite{kr}. The couplings between fluid field and radiation field involve momentum source and energy source depending on the specific radiation intensity driven by the so-called radiation  transfer  equation \cite{gp}.
Suppose that the matter is in local thermodynamical equilibrium (LTE), the coupled system  of Navier-Stokes-Boltzmann (RHD) equations  for the mass density $\rho(t,x)$, the velocity $u(t,x)=(u^{(1)},u^{(2)},u^{(3)})$ of the fluid and the specific radiation intensity $I(v,\Omega,t,x)$ in three-dimensional space reads as \cite{gp}
\begin{equation}
\label{eq:1.1}
\begin{cases}
\displaystyle
\frac{1}{c}I_t+\Omega\cdot\nabla I=A_r,\\[10pt]
\displaystyle
\rho_t+\text{div}(\rho u)=0,\\[10pt]
\displaystyle
\left(\rho u+\frac{1}{c^{2}}F_r\right)_t+\text{div}(\rho u\otimes u+P_r)
  +\nabla p_m =\text{div}\mathbb{T},
\end{cases}
\end{equation}
where  $t\geq 0$ and $x\in \mathbb{R}^3$ are the time and space variables, respectively. $p_m$ is the material pressure satisfying the  equation of state:
\begin{equation}
\label{eq:1.3}
p_m=A\rho^{\gamma},
\end{equation}
where $A>0$ and $\gamma>1 $ are both constants, $\gamma $ is the adiabatic exponent.
$\mathbb{T}$ is the viscosity stress tensor given by
\begin{equation}
\label{eq:1.4}
\mathbb{T}=\mu(\nabla u+(\nabla u)^\top)+\lambda \text{div}u\,\mathbb{I}_3,
\end{equation}
where $\mathbb{I}_3$ is the $3\times 3$ unit matrix,
$\mu$ is the shear viscosity coefficient, $\lambda+\frac{2}{3}\mu$ is the bulk viscosity coefficient, $\mu$ and $\lambda$  are both real constants satisfying
\begin{equation}
\label{eq:1.5}
  \mu > 0, \quad \lambda+\frac{2}{3}\mu \geq 0
  \end{equation}
which ensure the ellipticity of the Lam$\acute{  \text{e} }$ operator defined by
\begin{equation}\label{lame}
\displaystyle
Lu=-\text{div}\mathbb{T}=-\mu\triangle u-(\lambda+\mu)\nabla \text{div} u.
\end{equation}
$v$ and $\Omega$ are radiation variables.  $v\in \mathbb{R}^+$ is the frequency of photon, and $\Omega\in S^2 $ is the travel direction of photon. The radiation flux $F_r$ and  the radiation pressure tensor $P_r$ are defined by
\begin{equation*}
\displaystyle
F_r=\int_0^\infty \int_{S^{2}}  I(v,\Omega,t,x)\Omega \text{d}\Omega \text{d}v,\
P_r=\frac{1}{c}\int_0^\infty \int_{S^{2}}  I(v,\Omega,t,x)\Omega\otimes\Omega\text{d}\Omega \text{d}v,
\end{equation*}
where $S^2$ is the unit sphere in $\mathbb{R}^3$. The collision term on the right-hand side of the radiation transfer equation is
$$
A_r=S-\sigma_aI
+\int_0^\infty \int_{S^{2}} \Big(\frac{v}{v'}\sigma_sI'
-\sigma'_sI\Big) \text{d}\Omega' \text{d}v',
$$
where $I=I(v,\Omega,t,x),\,\,I'=I(v',\Omega',t,x)$; $S=S(v,\Omega,t,x)\geq 0$ is the rate of energy emission due to spontaneous process; $\sigma_a=\sigma_a(v,\Omega,t,x,\rho)\geq 0$ denotes the absorption coefficient that may also depend on the mass density $\rho$; $\sigma_s$ is the ``differential scattering coefficient'' such that the probability of a photon being
scattered from $v'$ to $v$ contained in $\text{d}v$,  from $\Omega'$ to
$\Omega$ contained in $\text{d}\Omega$, and travelling a distance $\text{d}s$ is
given by $\sigma_s(v' \rightarrow v,\Omega'\cdot\Omega)\text{d}v \text{d}\Omega
\text{d}s$, and
\begin{equation*}
\begin{split}
\sigma_s\equiv  \sigma_s(v' \rightarrow v, \Omega'\cdot\Omega,\rho)=O(\rho),\ \sigma'_s\equiv  \sigma_s(v \rightarrow v', \Omega\cdot\Omega',\rho)=O(\rho).
\end{split}
\end{equation*}

When there is no radiation effect, the local existence of strong solutions with vacuum has been solved by many authors, we refer the reader to \cite{CK2}\cite{CK3}\cite{CK}. Huang-Li-Xin \cite{HX1} obtained the well-posedness of classical solutions with large oscillations and vacuum for Cauchy problem \cite{HX1} to the isentropic flow. 

In general, the study of radiation hydrodynamics equations is challenging due to the high complexity and mathematical difficulty of the equations themselves. For the Euler-Boltzmann equations of the inviscid compressible radiation fluid, Jiang-Zhong \cite{sjx} obtained the local existence of $C^1$ solutions for the Cauchy problem  away from vacuum. Jiang-Wang \cite{pjd} showed that some $ C^1$ solutions  will blow up in  finite time, regardless of the size of the initial disturbance. Li-Zhu \cite{sz1} established the local existence of Makino-Ukai-Kawashima type (see \cite{tms1}) regular solutions with vacuum, and also proved that the regular solutions will blow up if the initial mass density vanishes in some local domain.

For the Navier-Stokes-Boltzmann equations of the viscous compressible radiation fluid, under some physical assumptions, Chen-Wang \cite{zcy} studied the classical solutions of the Cauchy problem with the mass density away from vacuum. Ducomet and  Ne$\check{\text{c}}$asov$\acute{\text{a}}$ \cite{BD}\cite{BS}  obtained  the global weak solutions  and their large time behavior for the one-dimensional case. Li-Zhu \cite{sz2} considered the formation of singularities on classical solutions in multi-dimensional space ($d\geq 2$),  when the initial mass density is compactly supported and the initial specific radiation intensity satisfies some directional condtions. Some special phenomenon has been observed, for example, it is known in contrast with the second law of thermodynamics, the associated entropy equation may contain a negative production term for RHD system, which has already been observed in Buet and Despr$\acute{\text{e}}$s \cite{add1}. Moreover, from  Ducomet, Feireisl and Ne$\check{\text{c}}$asov$\acute{\text{a}}$ \cite{add2},  in which they  obtained the existence of global weak solution for some RHD model,  we know that the velocity field $u$ may develop uncontrolled time oscillations on the hypothetical vacuum zones.

The purpose of this paper is to provide a local theory of strong solutions (see Definition \ref{strong1}) to the RHD equations in the framework of Sobolev spaces.
Via the radiation transfer equation $(\ref{eq:1.1})_1$ and  the definitions of $F_r$ and $P_r$,  system (\ref{eq:1.1}) can be rewritten as
\begin{equation}
\label{eq:1.2}
\begin{cases}
\displaystyle
\frac{1}{c}I_t+\Omega\cdot\nabla I=A_r,\\[10pt]
\displaystyle
\rho_t+\text{div}(\rho u)=0,\\[10pt]
\displaystyle
(\rho u)_t+\text{div}(\rho u\otimes u)
  +\nabla p_m +Lu=-\frac{1}{c}\int_0^\infty \int_{S^2}A_r\Omega \text{d}\Omega \text{d}v,
\end{cases}
\end{equation}
where $L$ is the Lam$\acute{  \text{e} }$ operator defined by \eqref{lame}.
We consider the Cauchy problem of \eqref{eq:1.2} with the following  initial data
\begin{equation} \label{eq:2.2hh}
I|_{t=0}=I_0(v,\Omega,x),\quad (\rho, u)|_{t=0}=(\rho_0(x), u_0(x)),\ (v,\Omega,x)\in \mathbb{R}^+\times S^2\times \mathbb{R}^3.
\end{equation}

For (\ref{eq:1.2})-(\ref{eq:2.2hh}), inspired by the argument used in \cite{CK2}\cite{CK},  we introduce a similar initial layer compatibility condition  (\ref{kkkkk}),  which will be used to compensate the loss of positive lower bound of the initial mass density when vacuum appears. The key point is to get a priori estimates independent of the lower bound of the initial mass density by this compatibility condition. Then the existence of the local strong solutions can be obtained by the approximation process from non-vacuum to vacuum. We also prove that if the initial vacuum is not so irregular, then the compatibility condition of the initial data is necessary and sufficient for the existence of a unique strong solution. Finally, we give a blow-up criterion for the local strong solution: if $\overline{T}< +\infty$ is the maximal existence time of the local strong solution $(I,\rho,u)$, then
\begin{equation*}
\begin{split}
\lim \sup_{t\mapsto \overline{T}} \big(\|I(t)\|_{L^2(\mathbb{R}^+\times S^2; H^1\cap W^{1,q}(\mathbb{R}^3))}+\|\rho(t)-\overline{\rho}\|_{H^1\cap W^{1,q}}+|u(t)|_{\mathbb{D}^1}\big)=+\infty,
\end{split}
\end{equation*}
 where $3< q\leq 6$ and $\overline{\rho}\geq 0$ are both constants.  The similar results also hold for the barotropic flow with general pressure law $p_m=p_m(\rho)\in C^1(\mathbb{\overline{R}}^+)$.

Throughout this paper, we use the following simplified notations for standard homogenous and inhomogeneous Sobolev spaces:
\begin{equation*}\begin{split}
&\|f(v,\Omega,t,x,\rho(t,x))\|_{X_1(\mathbb{R}^+\times S^2;X_2(\mathbb{R}^3))}=\big\|\|f(v,\Omega,t,\cdot,\rho(t,\cdot))\|_{X_2(\mathbb{R}^3)}\big\|_{X_1(\mathbb{R}^+\times S^2)},\\
&\|f(v,\Omega,t,x,\rho(t,x))\|_{X_1(\mathbb{R}^+\times S^2;X_2([0,T]\times \mathbb{R}^3))}
=\big\|\|f(v,\Omega,\cdot,\cdot,\rho(\cdot,\cdot))\|_{X_2([0,T]\times\mathbb{R}^3)}\big\|_{X_1(\mathbb{R}^+\times S^2)},\\
& \|(f,g)\|_X=\|f\|_X+\|g\|_X,\quad |\|g\||_{X,T}=|\|g(t,x)\||_{X,T}=\sup_{t\in [0,T]}\|g(t,\cdot)\|_{X},\\
& \|f\|_{W^{m,p}}=\|f\|_{W^{m,p}(\mathbb{R}^3)},\quad \|f\|_s=\|f\|_{H^s(\mathbb{R}^3)},\quad |f|_p=\|f\|_{L^p(\mathbb{R}^3)},\\
&D^{k,r}=\{f\in L^1_{loc}(\mathbb{R}^3): |\nabla^kf|_{r}<+\infty\},\,\, D^k=D^{k,2},\,\, \mathbb{D}^1=\{f\in L^6(\mathbb{R}^3): |\nabla f|_{2}<\infty\},\\
  & |f|_{D^{k,r}}:=\|f\|_{D^{k,r}(\mathbb{R}^3)}=|\nabla^kf|_{r},\ |f|_{D^{k}}:=\|f\|_{D^{k}(\mathbb{R}^3)}=|\nabla^kf|_{2}, \,\,|f|_{\mathbb{D}^1}:=\|f\|_{\mathbb{D}^1(\mathbb{R}^3)}=|\nabla f|_{2},
\end{split}
\end{equation*}
where  $0< T<\infty$ and $1\leq p \leq\infty$ are both constants, $X$, $X_1$, and $X_2$ are some Sobolev spaces.
The following inequalities will be used in our paper:
$$
|u|_6\leq C|u|_{\mathbb{D}^1},\quad |u|_{\infty}\leq C\|u\|_{\mathbb{D}^1\cap D^2}, \quad |u|_{\infty}\leq C\|u\|_{W^{1,q}},
$$
where  $3< q\leq 6$  and  $\|u\|_{X_1\cap X_2}=\|u\|_{X_1}+\|u\|_{X_2}$.  A detailed study on homogeneous Sobolev spaces may be found in \cite{gandi}.

Now we make some assumptions on the physical coefficients $\sigma_a$ and  $\sigma_s$. First, let
$$ \sigma_s=\overline{\sigma}_s(v' \rightarrow v, \Omega'\cdot\Omega)\rho,\quad \sigma'_s=\overline{\sigma}'_s(v \rightarrow v', \Omega\cdot\Omega')\rho,$$
where the functions $\overline{\sigma}_s\geq 0$ and $\overline{\sigma}'_s\geq 0$ satisfy
\begin{equation}\label{zhen1}
\begin{cases}
\displaystyle
\int_0^\infty \int_{S^{2}}\Big(\int_{0}^{\infty}\int_{S^{2}} \Big|\frac{v}{v'}\Big|^2\overline{\sigma}^2_s\text{d}\Omega' \text{d}v'\Big)^{\lambda_1}\text{d}\Omega \text{d}v\leq C,\\[12pt]
\displaystyle
\int_0^\infty \int_{S^{2}}\Big(\int_{0}^{\infty}\int_{S^{2}} \overline{\sigma}'_s\text{d}\Omega' \text{d}v'\Big)^{\lambda_2}\text{d}\Omega \text{d}v+\int_{0}^{\infty}\int_{S^{2}} \overline{\sigma}'_s\text{d}\Omega' \text{d}v'\leq C,
\end{cases}
\end{equation}
where $\lambda_1=1$ or $\frac{1}{2}$, and $\lambda_2=1$ or $2$. Hereinafter we denote by $C$ a generic positive constant depending only on the fixed constants $\mu$, $\lambda$, $\gamma$, $q$, $T$  and the norms of $S$.
Second, let
$$\sigma_a=\sigma(v,\Omega,t,x,\rho)\rho,$$
then for $\rho^{i}(t)\,(i=1,2)$ satisfying
\begin{equation}\label{zhen3}
\begin{split}
\|\rho^i(t)-\overline{\rho}\|_{H^1\cap W^{1,q}(\mathbb{R}^3)}+\|\rho^i_{t}(t)\|_{L^2\cap L^q(\mathbb{R}^3)}<& +\infty,\
\end{split}
\end{equation}
we assume that
\begin{equation}\label{jia345}
\begin{cases}
\|\sigma(v,\Omega,t,x,\rho^i)\|_{L^2\cap L^\infty(\mathbb{R}^+\times S^2; L^\infty(\mathbb{R}^3))}\leq M( |\rho^i(t)|_\infty),\\[6pt]
\|\nabla \sigma(v,\Omega,t,x,\rho^i)\|_{L^2\cap L^\infty(\mathbb{R}^+\times S^2; L^r(\mathbb{R}^3))}\leq M( |\rho^i(t)|_\infty) (|\nabla \rho^i(t)|_r+1),\\[6pt]
\|\sigma_t(v,\Omega,t,x,\rho^i)\|_{L^2 (\mathbb{R}^+\times S^2; L^2(\mathbb{R}^3))}
\leq  M(|\rho^i(t)|_\infty)(|\rho^i_t(t)|_2+1),\\[6pt]
 |\sigma(v,\Omega,t,x,\rho_1)-\sigma(v,\Omega,t,x,\rho_2)|\leq  \overline{\sigma}(v,\Omega,t,x,\rho_1,\rho_2)|\rho^1(t)-\rho^2(t)|, \\[6pt]
 \|\overline{\sigma}(v,\Omega,t,x,\rho_1,\rho_2)\|_{L^\infty\cap L^2(\mathbb{R}^+\times S^2;L^\infty( \mathbb{R}^3))}\leq M(|(\rho^1,\rho^2)(t)|_\infty),
\end{cases}
\end{equation}
for $t\in [0,T]$ and $r \in [2,q]$, where $M=M(\cdot):\,[0,+\infty) \rightarrow[1,+\infty) $ denotes a strictly increasing continuous function, and $\sigma(v,\Omega,t,x,\rho^i)\in C([0,T]; L^2(\mathbb{R}^+\times S^2; L^{\infty}(\mathbb{R}^3)))$.
\begin{remark}\label{con}
These assumptions are similar to those in \cite{sjx} for the local existence of classical solutions to the Euler-Boltzmann equations with initial mass density away form vacuum and the assumptions in \cite{sz1} for the local existence of regular solutions  with vacuum.  The evaluation of these radiation quantities is a difficult problem of  quantum mechanics, and their general forms are usually not known. The expressions of $\sigma_a$ and $\sigma_s$ used for describing Compton Scattering process in \cite{gp} are given by
\begin{equation}\label{kk}
\begin{split}
&\sigma_a(v,t,x,\rho,\theta)=D_1\rho \theta^{-\frac{1}{2}}\exp\Big(-\frac{D_2}{\theta^{\frac{1}{2}}}\Big(\frac{v-v_0}{v_0}\Big)^2\Big),\ \sigma_s=\overline{\sigma}_s(v \rightarrow v', \Omega\cdot\Omega')\rho,
\end{split}
\end{equation}
where $v_0$ is the fixed frequency, $D_i(i=1,2)$ are positive constants and $\theta$ is the temperature.
\end{remark}

The rest of  this paper is organized as follows. In Section $2$, we give our main results including the local existence of strong solutions with vacuum,  the necessity and sufficiency of the initial layer compatibility condition and  the corresponding blow-up criterion for the local strong solution that we obtained.
In Section $3$, we prove the existence and uniqueness of  local strong solutions via establishing a priori estimates independent of the lower bound of $\rho_0$. In Section $4$, we show that the initial layer compatibility condition is necessary and sufficient for the existence of a unique local strong solution.  Finally in Section $5$,  we prove the  blow-up criterion that we claimed in Section $2$. 

\section{Main results}

We state our main results in this section.
First, we give the definition of strong solutions to Cauchy problem (\ref{eq:1.2})-(\ref{eq:2.2hh}).
\begin{definition}[\textrm{Strong solutions}]\label{strong1}
$(I,\rho,u)$ is a strong solution on $\mathbb{R}^+\times S^2\times [0,T]\times\mathbb{R}^3 $ to Cauchy problem (\ref{eq:1.2})-(\ref{eq:2.2hh}) if the following holds:
\begin{enumerate}
\item  $(I,\rho,u)$ solves (\ref{eq:1.2})-(\ref{eq:2.2hh}) in the following sense of distribution:
\begin{equation}\label{weak}\begin{split}
&\int_0^\infty \int_{S^{2}}\int_{0}^{T}\int_{\mathbb{R}^3}\Big(\frac{1}{c}I\xi_t+I\Omega\cdot\nabla\xi\Big)\text{d}x \text{d}t\text{d}\Omega \text{d}v+\int_0^\infty \int_{S^{2}}\int_{\mathbb{R}^3}\frac{1}{c}I_0\xi(0)\text{d}x \text{d}\Omega \text{d}v\\
=&-\int_0^\infty \int_{S^{2}}\int_{0}^{T}\int_{\mathbb{R}^3}A_r\xi\text{d}x \text{d}t\text{d}\Omega \text{d}v;\\[6pt]
&\int_{0}^{T}\int_{\mathbb{R}^3}\Big(\rho\zeta_t+\rho u \cdot\nabla\zeta\Big)\text{d}x \text{d}t+\int_{\mathbb{R}^3}\rho_0\zeta(0)\text{d}x=0;\\[6pt]
&\int_{0}^{T}\int_{\mathbb{R}^3}\Big(\rho u\varphi_t+\rho u\otimes u:\nabla\varphi+p_m\text{div}\varphi-\mu \nabla u:\nabla \varphi-(\lambda+\mu)\text{div}u\text{div}\varphi \Big)\text{d}x \text{d}t\\
\displaystyle
=&-\int_{\mathbb{R}^3}\rho_0u_0\varphi(0)\text{d}x+\frac{1}{c}\int_{0}^{T}\int_{\mathbb{R}^3}A_r\Omega\cdot\varphi\text{d}x \text{d}t;
\end{split}
\end{equation}
for any test functions $\xi=\xi(v,\Omega,t,x)\in C^\infty_c(\mathbb{R}^+\times S^2\times [0,T)\times\mathbb{R}^3 )$, $\zeta=\zeta(t,x)\in C^\infty_c( [0,T)\times\mathbb{R}^3 )$ and  $\varphi \in \mathbb{R}^3$ with $\varphi \in  C^\infty_c( [0,T)\times\mathbb{R}^3 )$.
\item  $(I,\rho,u)$ satisfies the following regularities:
\begin{equation*}\begin{split}
&I\in L^2(\mathbb{R}^+\times S^2;C([0,T];H^1\cap W^{1,q}(\mathbb{R}^3))),\ I_t\in L^2(\mathbb{R}^+\times S^2;C([0,T];L^{2}\cap L^{q}(\mathbb{R}^3))),\\
&\rho\geq 0,\ \rho-\overline{\rho}\in C([0,T];H^1\cap W^{1,q}),\quad \rho_t\in C([0,T];L^{2}\cap L^{q}),\\
&u\in C([0,T];\mathbb{D}^1\cap D^2)\cap  L^2([0,T];D^{2,q}),\ u_t\in  L^2([0,T];\mathbb{D}^1), \ \sqrt{\rho}u_t\in  L^\infty([0,T];L^2).
\end{split}
\end{equation*}
\end{enumerate}
\end{definition}

As has been observed in $3$-D compressible isentropic Navier-Stokes equations \cite{CK}, in order to make sure that the Cauchy problem with initial mass density containing vacuum is well-posed, the lack of a positive lower bound of the initial mass density $\rho_0$ should be compensated by some initial layer compatibility condition on the initial data $(\rho_0,u_0)$. Now  considering the $3$-D compressible isentropic radiation hydrodynamic  equations (\ref{eq:1.2}), if we denote  $p^0_m=A\rho^\gamma_0$, $S_0=S(v,\Omega,t=0,x)$, and
\begin{equation*}
\begin{split} A^0_r=&S_0-\sigma_a(v,\Omega,t=0,x,\rho_0)I_0
+\int_0^\infty \int_{S^{2}} \Big(\frac{v}{v'}\sigma_s(\rho_0)I'_0
-\sigma'_s(\rho_0)I_0\Big) \text{d}\Omega' \text{d}v',\\
 \sigma_s(\rho_0)=&  \sigma_s(v' \rightarrow v, \Omega'\cdot\Omega,\rho_0),\ \sigma'_s(\rho_0)= \sigma_s(v \rightarrow v', \Omega\cdot\Omega',\rho_0),
\end{split}
\end{equation*}
then the   main result  of this paper on the existence of the unique  local strong solutions can be shown as
\begin{theorem} [{\textbf{Local existence of strong solutions}}]\ \\[4pt]
\label{th1}
Let the  assumptions (\ref{zhen1})-(\ref{jia345}) hold, and assume that
$$\|S(v,\Omega,t,x)\|_{ L^2(\mathbb{R}^+\times S^2;C^1([0,\infty);H^1\cap W^{1,q}(\mathbb{R}^3)))\cap C^1([0,\infty); L^1(\mathbb{R}^+\times S^2;L^1\cap L^2(\mathbb{R}^3)))}< +\infty.$$
If the initial data $(I_0,\rho_0, u_0)$ satisfy the regularities
\begin{equation}\label{gogo}
\begin{split}
& I_0(v,\Omega,x)\in L^2(\mathbb{R}^+\times S^2; H^1\cap W^{1,q}(\mathbb{R}^3)),\\
&\rho_0\geq 0, \quad \rho_0-\overline{\rho}\in H^1\cap W^{1,q},\quad u_0\in \mathbb{D}^1\cap D^2,\\
&(I_0, \rho_0,  u_0)\rightarrow (0,\overline{\rho},0), \ \text{as}\ |x|\mapsto\infty,\  \forall \  (v,\Omega)\in\mathbb{R}^+\times S^2,
\end{split}
\end{equation}
and the initial layer compatibility condition
\begin{equation}\label{kkkkk}
\begin{split}
Lu_0+\nabla p^0_{m}+\frac{1}{c}\int_0^\infty \int_{S^2}A^0_r\Omega \text{d}\Omega \text{d}v=&\rho^{\frac{1}{2}}_0 g_1
\end{split}
\end{equation}
for some $g_1 \in L^2$, then there exists a time $T_*>0$ and a unique strong solution $(I,\rho,u)$ on $\mathbb{R}^+\times S^2\times [0,T_*]\times\mathbb{R}^3 $ to Cauchy problem (\ref{eq:1.2})-(\ref{eq:2.2hh}).
\end{theorem}

\begin{remark}\label{zhen99}
For the case that the rate of energy emission $S$ depends on the mass density $\rho$, that is, $S=S(v,\Omega, t,x,\rho)$, similar results can be obtained via the same argument as the case $S=S(v,\Omega,t,x)$, if we assume, for $\rho^{i}(t) \ (i=1,2)$ satisfying (\ref{zhen3}), that,
\begin{equation}\label{jia666}
\begin{cases}
\|S(v,\Omega,t,x,\rho^i)\|_{ L^2(\mathbb{R}^+\times S^2;L^\infty( \mathbb{R}^3))\cap L^1(\mathbb{R}^+\times S^2;L^r( \mathbb{R}^3))} \leq M(|\rho^i(t)|_\infty),\\[6pt]
\|\nabla S(v,\Omega,t,x,\rho^i)\|_{ L^2(\mathbb{R}^+\times S^2;L^r( \mathbb{R}^3))}\leq M(|\rho^i(t)|_\infty)(|\nabla \rho(t)|_r+1),\\[6pt]
\|S_t(v,\Omega,t,x,\rho^i)\|_{ L^1(\mathbb{R}^+\times S^2;L^1\cap L^2( \mathbb{R}^3))} \leq M(|\rho^i(t)|_\infty)(|\rho_t(t)|_2+1),\\[6pt]
 |S(v,\Omega,t,x,\rho_1)-S(v,\Omega, t,x,\rho_2)|\leq  \overline{S}(v,\Omega,t,x,\rho_1,\rho_2)|\rho^1(t)-\rho^2(t)|, \\[6pt]
 \|\overline{S}(v,t,x,\rho_1,\rho_2)\|_{ L^1(\mathbb{R}^+\times S^2;L^3( \mathbb{R}^3))\cap  L^2 (\mathbb{R}^+\times S^2;L^2( \mathbb{R}^3))}\leq M(|(\rho^1,\rho^2)(t)|_\infty)
\end{cases}
\end{equation}
for $t\in [0,T]$ and $r \in [2,q]$.
\end{remark}

Our second result can be regarded as an explanation for  the compatibility  between  (\ref{gogo}) and  (\ref{kkkkk}) when the initial vacuum is not so irregular. To be more precise, we denote by $V$ the initial vacuum set, i.e, the interior of the zero-set of the initial density in $\mathbb{R}^3$, and define the Sobolev space $D^{1}_0(V)$ as
$$ D^{1}_0(V)=\{f\in L^6(V): |f|_{D^{1}_0}=|\nabla f|_{L^2}<\infty,\ f|_{\partial V}=0\}.$$
Then we have
\begin{theorem} [{\textbf{Necessity and sufficiency of the compatibility condition}}]\ \\[4pt]
\label{th2}
Let conditions in Theorem \ref{th1} hold. We assume that either the initial vacuum set $V$ is empty or the elliptic system
\begin{equation}\label{zhen101}
L\phi=-\mu\triangle \phi-(\lambda+\mu)\nabla \emph{div}\phi=0
\end{equation}
has only zero solution in $\mathbb{D}^1_0(V)\cap D^2(V)$. Then there exists a unique local strong solution $(I,\rho,u)$ satisfying
\begin{equation}\label{coco}
\begin{split}
&\|I(t)-I_0\|_{H^1\cap W^{1,q}(\mathbb{R}^3)}\rightarrow 0, \quad \text{as}\ t\rightarrow 0, \ \forall \ (v,\Omega) \in \mathbb{R}^+\times S^2,\\
&\|\rho(t)-\rho_0\|_{H^1\cap W^{1,q}(\mathbb{R}^3)}+|u(t)-u_0|_{\mathbb{D}^1\cap D^2(\mathbb{R}^3)}\rightarrow 0, \quad \text{as}\ t\rightarrow 0,
\end{split}
\end{equation}
if and only if the initial data satisfy the initial layer compatibility condition (\ref{kkkkk}).
\end{theorem}
\begin{remark}\label{initial}
From the regularities of the strong solution $(I,\rho,u)$ in the Definition \ref{strong1}, we know that
\begin{equation*}
\begin{split}
&I(v,\Omega,t,x)\in L^2(\mathbb{R}^+\times S^2;C([0,T];H^1\cap W^{1,q}(\mathbb{R}^3))), \\
&\rho(t,x)-\overline{\rho}\in C([0,T];H^1\cap W^{1,q}(\mathbb{R}^3)), \\
&u(t,x)\in C([0,T];\mathbb{D}^1\cap D^2(\mathbb{R}^3)).
\end{split}
\end{equation*}
But since the strong solution $(I,\rho,u)$ satisfies the Cauchy problem only  in the sense of distribution, we only have $I(v,\Omega,t=0,x)=I_0$, $\rho(t=0,x)=\rho_0$ and $\rho u(t=0,x)=\rho_0 u_0$. In the vacuum domain, the relation $u(t=0,x)=u_0$ maybe not hold.
Theorem \ref{th2} tells us  that if the initial vacuum set $V$ has a sufficiently simple geometry, for instance, it is a domain with Lipschitz boundary, we then have $u(t=0,x)=u_0$.
\end{remark}
Finally, we give a blow-up criterion for  strong solutions obtained in Theorem \ref{th1}.
\begin{theorem} [{\textbf{Blow-up criterion for the local strong solution}}]\ \\[4pt]
\label{th3}
Let conditions in Theorem \ref{th1} hold.
If $\overline{T}< +\infty $ is the maximal existence time of the local strong solution $(I,\rho, u)$ obtained in Theorem \ref{th1}, then we have
\begin{equation}\label{blowcr}
\begin{split}
\lim \sup_{t\mapsto \overline{T}} \big(\|I(t)\|_{L^2(\mathbb{R}^+\times S^2; H^1\cap W^{1,q}(\mathbb{R}^3))}+\|\rho(t)-\overline{\rho}\|_{H^1\cap W^{1,q}}+|u(t)|_{\mathbb{D}^1}\big)=+\infty.
\end{split}
\end{equation}

\end{theorem}

\begin{remark}[\textbf{General barotropic flow}]
\label{co1}
Similar results also hold for general barotropic flow.
Let (\ref{zhen1}) and (\ref{jia345}) hold, $p_m=p_m(\rho)\in C^1(\mathbb{\overline{R}}^+)$ and
$$\|S(v,\Omega,t,x)\|_{ L^2(\mathbb{R}^+\times S^2;C^1([0,\infty);H^1\cap W^{1,q}(\mathbb{R}^3)))\cap C^1([0,\infty); L^1(\mathbb{R}^+\times S^2;L^1\cap L^2(\mathbb{R}^3)))}< +\infty.$$
Assume that the initial data $(I_0, \rho_0,u_0)$ satisfy the regularity conditions
\begin{equation}\label{gogggo}
\begin{split}
& I_0(v,\Omega,x)\in L^2(\mathbb{R}^+\times S^2; H^1\cap W^{1,q}(\mathbb{R}^3)), \\
&\rho_0\geq 0,\quad \rho_0-\overline{\rho}\in H^1\cap W^{1,q},\quad u_0\in \mathbb{D}^1\cap D^2,\\
&(I_0, \rho_0,  u_0)\rightarrow (0,\overline{\rho},0), \ \text{as}\ |x|\mapsto\infty,\  \forall (v,\Omega)\in\mathbb{R}^+\times S^2,
\end{split}
\end{equation}
and the compatibility condition
\begin{equation}\label{kkcv}
\begin{split}
Lu_0+\nabla p^0_{m}+\frac{1}{c}\int_0^\infty \int_{S^2}A^0_r\Omega \text{d}\Omega \text{d}v=&\rho^{\frac{1}{2}}_0 g_2
\end{split}
\end{equation}
for some $g_2 \in L^2$, where  $p^0_m=p_m(\rho_0)$,  $ A^0_r$ is defined as before.
Then the conclusions obtained in Theorems \ref{th1}-\ref{th3} also hold for (\ref{eq:1.2})-(\ref{eq:2.2hh}).
\end{remark}
\section{The existence and uniqueness of local strong solutions}

We prove Theorem \ref{th1} in this section, i.e., the existence and uniqueness of local strong solutions. For the rate of energy emission $S$,  we always assume that
$$\|S(v,\Omega, t,x)\|_{ L^2(\mathbb{R}^+\times S^2;C^1([0,\infty);H^1\cap W^{1,q}(\mathbb{R}^3)))\cap C^1([0,\infty); L^1(\mathbb{R}^+\times S^2;L^1\cap L^2(\mathbb{R}^3)))}< +\infty.$$
In order to prove
 the local existence of strong solutions to the nonlinear problem, we need to consider the  linearized system
\begin{equation}
\label{eq:1.weeer}
\begin{cases}
\displaystyle
\rho_t+\text{div}(\rho w)=0,\\[8pt]
\displaystyle
\frac{1}{c}I_t+\Omega\cdot\nabla I=\overline{A}_r,\\[8pt]
\displaystyle
(\rho u)_t+\text{div}(\rho w\otimes u)
  +\nabla p_m +Lu=-\frac{1}{c}\int_0^\infty \int_{S^2}A_r\Omega \text{d}\Omega \text{d}v,
\end{cases}
\end{equation}
with the initial data (\ref{eq:2.2hh}), where
$w=w(t,x)\in \mathbb{R}^3$ is a known vector, the terms $\overline{A}_{r}$ and $A_{r}$ are defined by
\begin{equation*}
\begin{split}
\overline{A}_r=&S-\sigma_a(\rho)I+\int_0^\infty \int_{S^2} \Big(\frac{v}{v'}\sigma_s(\rho)\psi-\sigma'_s(\rho)I \Big)\text{d}\Omega' \text{d}v',\\
A_r=&S-\sigma_a(\rho)I+\int_0^\infty \int_{S^2} \Big(\frac{v}{v'}\sigma_s(\rho)I'-\sigma'_s(\rho)I \Big)\text{d}\Omega' \text{d}v',
\end{split}
\end{equation*}
where $\psi=\psi(v',\Omega',t,x)$ is a known function. We assume that
\begin{equation}\label{tyuu}
\begin{split}
& I_0(v,\Omega,x)\in L^2(\mathbb{R}^+\times S^2; H^1\cap W^{1,q}),\ \rho_0\geq 0,\ \rho_0-\overline{\rho}\in H^1\cap W^{1,q}, \ u_0\in \mathbb{D}^1\cap D^2,\\
& w\in C([0,T];\mathbb{D}^1\cap D^2)\cap  L^2([0,T];D^{2,q}), \ w_t\in  L^2([0,T];\mathbb{D}^1),\\
&\psi\in L^2(\mathbb{R}^+\times S^2;C([0,T];H^1\cap W^{1,q})),\ \psi_t\in L^2(\mathbb{R}^+\times S^2;C([0,T];L^{2}\cap L^{q})),\\
& (w,\psi)|_{t=0}=(u_0,I_0).
\end{split}
\end{equation}

\subsection{A priori estimates to the linearized problem away from vacuum}\ \\

We immediately  have the global existence of a unique strong solution $(I,\rho,u)$ to (\ref{eq:1.weeer}) with (\ref{eq:2.2hh}) by the standard methods at least for the case that the initial mass density is away from vacuum.

 \begin{lemma}\label{lem1}
 Assume in addition to (\ref{tyuu}) that $\rho_0\geq \delta$ for some constant $\delta>0$ and the compatibility condition (\ref{kkkkk}) holds.
Then there exists a unique strong solution $(I,\rho,u)$ to Cauchy problem (\ref{eq:1.weeer}) with (\ref{eq:2.2hh}) such that
\begin{equation*}\begin{split}
&I\in L^2(\mathbb{R}^+\times S^2;C([0,T];H^1\cap W^{1,q}(\mathbb{R}^3)))\cap C([0,T]; L^2(\mathbb{R}^+\times S^2; L^{2}\cap L^{q}(\mathbb{R}^3))),\\
&I_t\in L^2(\mathbb{R}^+\times S^2;C([0,T];L^{2}\cap L^{q}(\mathbb{R}^3))),\\
&\rho-\overline{\rho}\in C([0,T];H^1\cap W^{1,q}),\ \rho_t\in C([0,T];L^{2}\cap L^{q}),\ \rho \geq \underline{\delta},\\
&u\in C([0,T];H^2)\cap  L^2([0,T];D^{2,q}),\ u_t\in  C([0,T];L^2)\cap L^2([0,T];H^1),\ u_{tt}\in  L^2([0,T];H^{-1}),
\end{split}
\end{equation*}
for some constants $3<q \leq 6$ and $\underline{\delta}>0$.
\end{lemma}
\begin{proof}
First, the existence and regularity of the unique solution $\rho$ to  $(\ref{eq:1.weeer})_1$ can be obtained essentially according to the same argument in \cite{CK} for Navier-Stokes equations, and $\rho$ can be expressed by
\begin{equation}
\label{eq:bb1}
\rho(t,x)=\rho_0(U(0;t,x))\exp\Big(-\int_{0}^{t}\textrm{div}\, w(s,U(s,t,x))\text{d}s\Big),
\end{equation}
where  $U\in C([0,T]\times[0,T]\times \mathbb{R}^3)$ is the solution to the initial value problem
\begin{equation}
\label{eq:bb1}
\begin{cases}
\displaystyle\frac{\text{d}}{\text{d}s}U(s;t,x)=w(s,U(s;t,x)),\quad 0\leq s\leq T,\\
U(t;t,x)=x, \quad \ \qquad \quad 0\leq t\leq T,\ x\in \mathbb{R}^3,
\end{cases}
\end{equation}
so we can easily get the positive lower bound of $\rho$.

Second,   $(\ref{eq:1.weeer})_2$ can be rewritten into
\begin{equation}\label{rst}
\frac{1}{c}I_t+\Omega\cdot\nabla I+\Big(\sigma_a+\int_0^\infty \int_{S^2} \sigma'_s(\rho)\text{d}\Omega' \text{d}v'\Big)I=F(v,\Omega,t,x),
\end{equation}
where
$$F=S+\int_0^\infty \int_{S^2} \frac{v}{v'}\sigma_s(\rho)\psi\text{d}\Omega' \text{d}v'\in L^2(\mathbb{R}^+\times S^2;C([0,T];H^1\cap W^{1,q}(\mathbb{R}^3))),$$
then we easily get the existence and regularity of a unique solution $I$ to (\ref{rst}) such that
 $$
 I\in L^2(\mathbb{R}^+\times S^2;C([0,T];H^1\cap W^{1,q}(\mathbb{R}^3))),\ I_t\in L^2(\mathbb{R}^+\times S^2;C([0,T];L^{2}\cap L^{q}(\mathbb{R}^3))),
 $$
and according to the classical imbedding theory for Sobolev spaces, it is easy to show that
$$ I\in C([0,T]; L^2(\mathbb{R}^+\times S^2; L^{2}\cap L^{q}(\mathbb{R}^3))).$$

Finally,  the momentum equations   $(\ref{eq:1.weeer})_3$ can be written into
\begin{equation}\label{rst1}
\displaystyle
u_t+ w\cdot\nabla u
 +\rho^{-1}Lu= -\rho^{-1}\nabla p_m -\frac{1}{c}\rho^{-1}\int_0^\infty \int_{S^2}A_r\Omega \text{d}\Omega \text{d}v,
\end{equation}
then the existence and regularity of the unique solution $u$ to the corresponding linear parabolic problem can be obtained by standard methods as in \cite{CK2}\cite{CK3}.
\end{proof}

In order to pass to the limit as  $\delta\rightarrow 0$, we need to establish a priori estimates independent of $\delta$ for the  solution  $(I,\rho,u)$ to Cauchy problem (\ref{eq:1.weeer}) with (\ref{eq:2.2hh}) obtained in Lemma \ref{lem1}.

We fix a positive constant $c_0$ sufficiently large such that
\begin{equation*}\begin{split}
&2+\|\rho_0-\overline{\rho}\|_{H^1\cap W^{1,q}}+|u_0|_{\mathbb{D}^1\cap D^2}+\|I_0\|_{L^2(\mathbb{R}^+\times S^2; H^1\cap W^{1,q}(\mathbb{R}^3))}+|g_1|_2\\
&+\|S(v,\Omega, t,x)\|_{ L^2(\mathbb{R}^+\times S^2;C^1([0,\infty);H^1\cap W^{1,q}(\mathbb{R}^3)))\cap C^1([0,\infty); L^1(\mathbb{R}^+\times S^2;L^1\cap L^2(\mathbb{R}^3)))}\leq c_0,
\end{split}
\end{equation*}
and 
\begin{equation*}\begin{split}
\displaystyle
\|\psi\|_{L^2(\mathbb{R}^+\times S^2;C([0,T^*];H^1\cap W^{1,q}(\mathbb{R}^3)))}\leq& c_1,\\
\displaystyle
 \|\psi_t\|_{L^2(\mathbb{R}^+\times S^2;C([0,T^*];L^{2}\cap L^{q}(\mathbb{R}^3)))}\leq& c_2,\\
\displaystyle
\sup_{0\leq t \leq T^*}|w(t)|^2_{\mathbb{D}^1}+\int_{0}^{T^*}|w(s)|^2_{D^{2}}\text{d}s\leq& c^2_3,\\
\displaystyle
\sup_{0\leq t \leq T^*}|w(t)|^2_{D^2}+\int_{0}^{T^*}\Big(|w(s)|^2_{D^{2,q}}+|w_t(s)|^2_{\mathbb{D}^1}\Big)\text{d}s\leq& c^2_4,
\end{split}
\end{equation*}
for some time $T^*\in (0,T)$ and constants $c_i$ ($i=1,2,3,4$) such that 
$$1< c_0\leq c_1 \leq c_2 \leq c_3 \leq c_4. $$ The constants $c_i$ ($i=1,2,3,4$) and $T^*$ will be determined later and depend only on $c_0$ and the fixed constants $\overline{\rho}$, q, A, $\mu$, $\lambda$, $\gamma$, $c$ and $T$.
  As defined in assumption (\ref{jia345}), $M=M(\cdot): [0,+\infty) \rightarrow[1,+\infty)$ still denotes a strictly increasing continuous function depending only on fixed constants $\overline{\rho}$, q, A, $\mu$, $\lambda$, $\gamma$, $c$ and $T$.\\

We first give the a priori estimates  for density $\rho$. Hereinafter, we use  $C\geq 1$ to denote  a generic positive constant depending only on fixed constants $\overline{\rho}$, q, A, $\mu$, $\lambda$, $\gamma$, $c$ and $T$. 

\begin{lemma}[\textbf{Estimates for the mass density $\rho$}]\label{lem:2}
For the strong solution $(I,\rho,u)$ to the Cauchy problem (\ref{eq:1.weeer}) with (\ref{eq:2.2hh}), there exists a time $T_1>0$ such that
\begin{equation*}\begin{split}
\|\rho(t)-\overline{\rho}\|_{H^1\cap W^{1,q}}+\|p_m(t)-\overline{p}\|_{H^1\cap W^{1,q}}\leq& M(c_0),\\
 |\rho_t(t)|_{2}+ |(p_m)_t(t)|_{2}\leq M(c_0)c_3,\quad  |\rho_t(t)|_{q}+ |(p_m)_t(t)|_{q}\leq& M(c_0)c_4,
\end{split}
\end{equation*}
for  $0\leq t \leq T_1=\min(T^*,(1+c^2_4)^{-1})$ and $\overline{p}=A\overline{\rho}^\gamma$.
 \end{lemma}
\begin{proof}
From the continuity equation and  the standard energy estimates as shown in \cite{CK}, for $2\leq r\leq q$, we have
\begin{equation*}\begin{split}
\|\rho(t)-\overline{\rho}\|_{W^{1,r}}\leq \Big(\|\rho_0-\overline{\rho}\|_{W^{1,r}}+\int_0^t \|\nabla w(s)\|_{W^{1,r}}\text{d}s\Big) \exp\Big(C\int_0^t \|\nabla w(s)\|_{W^{1,q}}\text{d}s\Big).
\end{split}
\end{equation*}
Therefore,  the desired estimate for $\rho$ follows by observing that
$$
\int_0^t \|\nabla w(s)\|_{W^{1,r}}\text{d}s\leq t^{\frac{1}{2}}\Big(\int_0^t \|\nabla w(s)\|^2_{ W^{1,r}}\text{d}s\Big)^{\frac{1}{2}}\leq C(c_1t+(c_1t)^{\frac{1}{2}}),
$$
for  $0\leq t \leq T_1=\min(T^*,(1+c^2_4)^{-1})$. The estimate for $\rho_t$ is clear from $\rho_t=-\text{div}(\rho w)$.

Due to $p_m=A\rho^\gamma$ ($\gamma>1$), 
then the  estimate for $p_m$ follows immediately from above.
\end{proof}

Now we give the a priori estimates for $I$.
\begin{lemma}[\textbf{Estimates for specific radiation intensity $I$}]\label{lem:3}
For the strong solution $(I,\rho,u)$ to the Cauchy problem (\ref{eq:1.weeer}) with (\ref{eq:2.2hh}), there exists a time $T_2>0$ such that
\begin{equation}\label{2I}
\begin{split}
\|I\|_{L^2(\mathbb{R}^+\times S^2;C([0,T_2];H^1\cap W^{1,q}(\mathbb{R}^3)))}\leq &Cc_0,\\
 \|I_t\|_{L^2(\mathbb{R}^+\times S^2;C([0,T_2];L^{2}\cap L^{q}(\mathbb{R}^3)))}\leq& M(c_0)c_0,
\end{split}
\end{equation}
for  $T_2= \min(T^*,(1+M(c_0)c^2_4)^{-1})$.
 \end{lemma}

\begin{proof}Let $2\leq r\leq q$. First, multiplying $ (\ref{eq:1.weeer})_2$ by $r|I|^{r-2}I$ and integrating over $\mathbb{R}^3$ with respect to $x$, we have
\begin{equation}\label{thyt}
\begin{split}
\frac{\text{d}}{\text{d}t}|I|_{r}
\leq  C|S|_{r}
+C|\rho|_\infty\int_0^\infty \int_{S^2} \frac{v}{v'}|\psi|_{r} \overline{\sigma}_s \text{d}\Omega' \text{d}v',
\end{split}
\end{equation}
where we  used the fact that $\sigma_a\geq 0$ and $\sigma'_s\geq 0$.
According to the assumptions (\ref{zhen1})-(\ref{jia345}) and  H\"older's inequality, it is not hard to deduce that
\begin{equation}\label{thyt1}
\begin{split}
\frac{\text{d}}{\text{d}t}|I|^2_{r}
\leq& C\Big(|I|^2_{r}+|S|^2_{r}+ |\psi|^2_{L^2(\mathbb{R}^+\times S^2; L^r)}|\rho|^2_\infty  \int_0^\infty \int_{S^2}\Big|\frac{v}{v'}\Big|^2 \overline{\sigma}^2_s\text{d}\Omega' \text{d}v'\Big)\\
\leq& C\Big(|I|^2_{r}+|S|^2_{r}
+M(c_0)c^2_1\int_0^\infty \int_{S^2}\Big|\frac{v}{v'}\Big|^2 \overline{\sigma}^2_s\text{d}\Omega' \text{d}v'\Big),
\end{split}
\end{equation}
for $0\leq t\leq T_1$.

Second, differentiating  $ (\ref{eq:1.weeer})_2$ $\beta$-times ($|\beta|=1$) with respect to $x$, then multiplying the resulting equation by $r|\partial^\beta_xI|^{r-2}\partial^\beta_xI$ and  integrating over $\mathbb{R}^3$ with respect to $x$,  we get
\begin{equation}\label{thytffg}
\begin{split}
\frac{\text{d}}{\text{d}t}|\partial^\beta_x I|_{r}
\leq& C\big(|\partial^\beta_x S|_{r}+|\sigma_a|_{D^{1,q}}|I|_{\frac{qr}{q-r}}\big)+C|\nabla\rho|_{q}|I|_{\frac{qr}{q-r}}\int_0^\infty \int_{S^2} \overline{\sigma}'_s\text{d}\Omega' \text{d}v'\\
&+C \int_0^\infty \int_{S^2} \frac{v}{v'}\Big(|\rho|_{\infty}|\psi|_{D^{1,r}} \overline{\sigma}_s+
 |\rho|_{D^{1,r}} \|\psi\|_{W^{1,q}} \overline{\sigma}_s\Big)\text{d}\Omega' \text{d}v',
\end{split}
\end{equation}
where we also  used the fact that $\sigma_a\geq 0$ and $\sigma'_s\geq 0$.

Since $r<\frac{qr}{q-r}\leq \frac{3r}{3-r} $ if $2\leq r \leq 3$ and $r<\frac{qr}{q-r}\leq +\infty $ if $3<r \leq q$, it follows from Sobolev's imbedding theorem that $|I|_{\frac{qr}{q-r}}\leq C \|I\|_{W^{1,r}}$. From  assumption (\ref{jia345}) and Lemma \ref{lem:2}, we easily have
\begin{equation}\label{da1}
\begin{split}
|\sigma_a|_{D^{1,q}}\leq |\nabla \sigma|_q |\rho|_\infty+|\nabla\rho|_q |\sigma|_\infty\leq M( |\rho|_\infty)\big(|\rho|_\infty|\nabla\rho|_q+|\rho|_\infty+|\nabla\rho|_q\big).
\end{split}
\end{equation}

According to  assumptions (\ref{zhen1})-(\ref{jia345}), estimates (\ref{thytffg})-(\ref{da1}), and the H\"older's inequality,  we obtain
\begin{equation}\label{thytq}
\begin{split}\frac{\text{d}}{\text{d}t}|\partial^\beta_x I|^2_{r}
\leq&  C\big(1+|\sigma_a|_{D^{1,q}}+|\nabla\rho|_{q}\int_0^\infty \int_{S^2} \overline{\sigma}'_s\text{d}\Omega' \text{d}v'\big)\|I\|^2_{W^{1,r}}+C|\partial^\beta_x S|^2_{r}\\
&+M(c_0)\int_0^\infty \int_{S^2} \Big|\frac{v}{v'}\Big|^2 \overline{\sigma}^2_s \text{d}\Omega' \text{d}v'\cdot\int_0^\infty \int_{S^2} |\psi|^2_{D^{1,r}}  \text{d}\Omega' \text{d}v'\\
&+M(c_0)\int_0^\infty \int_{S^2} \Big|\frac{v}{v'}\Big|^2\overline{\sigma}^2_s\text{d}\Omega' \text{d}v'\cdot\int_0^\infty \int_{S^2} \|\psi\|^2_{W^{1,q}} \text{d}\Omega' \text{d}v'\\
\leq&M(c_0)\|I\|^2_{W^{1,r}}+C|\partial^\beta_x S|^2_{r}+M(c_0)c^2_1\int_0^\infty \int_{S^2} \Big|\frac{v}{v'}\Big|^2 \overline{\sigma}^2 _s \text{d}\Omega' \text{d}v'
\end{split}
\end{equation}
for $0\leq t\leq T_1$.
Then combining (\ref{thyt1}) and (\ref{thytq}), it turns out that
\begin{equation}\label{thytgb}
\begin{split}
&\frac{\text{d}}{\text{d}t}\|I\|^2_{W^{1,r}}
\leq M(c_0)\|I\|^2_{W^{1,r}}+C\| S\|^2_{W^{1,r}}+M(c_0)c^2_1\int_0^\infty \int_{S^2} \Big|\frac{v}{v'}\Big|^2 \overline{\sigma}^2_s \text{d}\Omega' \text{d}v'.
\end{split}
\end{equation}
From Gronwall's inequality, we get
\begin{equation}\label{coti}
\begin{split}
&\|I(v,\Omega,\cdot,\cdot)\|^2_{C([0,T_2];H^1\cap W^{1,q}(\mathbb{R}^3))}
\leq \exp (M(c_0)T_2)\|I_0\|^2_{W^{1,r}}\\
&+\exp (M(c_0)T_2)\Big(\int_0^{T_2}\|S\|^2_{W^{1,r}}\text{d}s+M(c_0)c^2_1T_2\int_0^\infty \int_{S^2} \Big|\frac{v}{v'}\Big|^2 \overline{\sigma}^2_s \text{d}\Omega' \text{d}v'\Big),
\end{split}
\end{equation}
where $ T_2= \text{min}(T^*,(1+M(c_0)c^2_4)^{-1})$.

Then integrating the above inequality in $\mathbb{R}^+\times S^2$ with respect to $(v,\Omega)$ and using assumptions (\ref{zhen1})-(\ref{jia345}), we arrive at
$$
\|I(v,\Omega,t,x)\|^2_{L^2(\mathbb{R}^+\times S^2;C([0,T_2];H^1\cap W^{1,q}(\mathbb{R}^3)))}\leq Cc^2_0.
$$
Finally, due to $ I_t=-c\Omega\cdot\nabla I+c\overline{A}_r$, the desired estimates for $I_t$ are obvious.
\end{proof}

Now we give the a priori estimates for $u$.
\begin{lemma}[\textbf{Estimates for velocity $u$}]\label{lem:4}
For the strong solution $(I,\rho,u)$ to the Cauchy problem (\ref{eq:1.weeer}) with (\ref{eq:2.2hh}), there exists a time $T_3$ such that 
\begin{equation}\label{gai}
\begin{split}
|u(t)|^2_{\mathbb{D}^1}+\int_{0}^{t}\Big(|u(s)|^2_{D^2}+|\sqrt{\rho}u_t(s)|^2_{2}\Big)\text{d}s\leq& M(c_0)c^{2}_0,\\
|u(t)|^2_{\mathbb{D}^1\cap D^2}+|\sqrt{\rho}u_t(t)|^2_{2}+\int_{0}^{t}\Big(|u(s)|^2_{D^{2,q}}+|u_t(s)|^2_{\mathbb{D}^1}\Big)\text{d}s\leq& M(c_0)c^{10}_3,
\end{split}
\end{equation}
for  $0\leq t \leq  T_3= \min(T^*,(1+M(c_0)c^{8}_4)^{-1})$.
\end{lemma}
\begin{proof} 
$\text{\underline{Step 1}}$. The estimate of $|u|_{\mathbb{D}^1}$.  Multiplying   $ (\ref{eq:1.weeer})_3$ by $u_t$ and integrating  over $\mathbb{R}^3$, we have
\begin{equation}\label{f1}
\begin{split}
&\int_{\mathbb{R}^3}\rho |u_t|^2 \text{d}x+\frac{1}{2}\frac{d}{dt}\int_{\mathbb{R}^3}\Big(\mu|\nabla u|^2+\big(\lambda+\mu\big)(\text{div}u)^2\Big) \text{d}x\\
=& \int_{\mathbb{R}^3} \Big(-\nabla p_m-\rho w\cdot \nabla u\Big)\cdot u_t\text{d}x-\frac{1}{c}\int_{\mathbb{R}^3} \int_0^\infty \int_{S^2}  A_r u_t\cdot \Omega\text{d}\Omega \text{d}v\text{d}x\\
=&\frac{d}{dt}\Lambda_1(t)- \Lambda_2(t)+E_I,
\end{split}
\end{equation}
where
\begin{equation*}
\begin{split}
\Lambda_1(t)=\int_{\mathbb{R}^3} ( p_m-\overline{p})\text{div}u\text{d}x,\quad \Lambda_2(t)=\int_{\mathbb{R}^3} \Big((p_m)_t\text{div}u+\rho (w\cdot \nabla u)\cdot u_t \Big)\text{d}x.
\end{split}
\end{equation*}
According to Lemma \ref{lem:2}, H\"older's inequality, Gagliardo-Nirenberg inequality and Young's inequality, we have
\begin{equation*}
\begin{split}
\Lambda_1(t)\leq& C|\nabla u|_2 |p_m-\overline{p}|_2
\leq M(c_0)|\nabla u|_2,\\
\Lambda_2(t)\leq&  C\big(|\nabla u|_2 |(p_m)_t|_2+|\rho|^{\frac{1}{2}}_{\infty} |\sqrt{\rho}u_t|_2  |w|_{\infty}  |\nabla u|_2\big)\\
\leq&  M(c_0)c_3|\nabla u|_2+M(c_0)c^2_4|\nabla u|^2_2+\frac{1}{10}|\sqrt{\rho} u_t|^2_2,
\end{split}
\end{equation*}
for $0< t\leq T_2$.

Now we estimate the radiation term $E_{I}$, where
\begin{equation*}
\begin{split}
E_{I}=&-\frac{1}{c}\int_{\mathbb{R}^3} \int_0^\infty \int_{S^2}  A_r u_t\cdot \Omega\text{d}\Omega \text{d}v\text{d}x \\
=&-\frac{1}{c}\int_0^\infty \int_{S^2} \int_{\mathbb{R}^3} \Big(S-\sigma_a I+ \int_0^\infty \int_{S^2} \frac{v}{v'}\sigma_s I'_t\text{d}\Omega'\text{d}v'\Big)u_t \cdot \Omega\text{d}x \text{d}\Omega \text{d}v\\
&+\frac{1}{c}\int_0^\infty \int_{S^2}\int_0^\infty \int_{S^2} \int_{\mathbb{R}^3} \sigma'_s I u_t\cdot \Omega\text{d}x \text{d}\Omega'\text{d}v' \text{d}\Omega \text{d}v
=:\sum_{j=1}^{4} J_j.
\end{split}
\end{equation*}
We estimate $J_{j}$ term by term. From  Lemmas \ref{lem:2}-\ref{lem:3}, Gagliardo-Nirenberg inequality, H\"older's inequality, Young's inequality and  (\ref{zhen1})-(\ref{jia345}), for $0\leq t\leq T_2$, we have

\begin{equation*}
\begin{split}
J_1= &-\frac{1}{c} \int_0^\infty \int_{S^2} \int_{\mathbb{R}^3}  S u_t\cdot \Omega \text{d}x \text{d}\Omega \text{d}v\\
=&-\frac{1}{c} \frac{d}{dt}\int_0^\infty \int_{S^2} \int_{\mathbb{R}^3}  S u\cdot \Omega \text{d}x \text{d}\Omega \text{d}v+\frac{1}{c} \int_0^\infty \int_{S^2} \int_{\mathbb{R}^3}  S_t u\cdot \Omega \text{d}x \text{d}\Omega \text{d}v\\
\leq& -\frac{1}{c} \frac{d}{dt}\int_0^\infty \int_{S^2} \int_{\mathbb{R}^3}  S u\cdot \Omega \text{d}x \text{d}\Omega \text{d}v+C|u|_{\mathbb{D}^1}\int_0^\infty \int_{S^2} |S_t |_{\frac{6}{5}}\text{d}\Omega \text{d}v\\
\leq& -\frac{1}{c} \frac{d}{dt}\int_0^\infty \int_{S^2} \int_{\mathbb{R}^3}  S u\cdot \Omega \text{d}x \text{d}\Omega \text{d}v+Cc_0|  \nabla u|_2,\\
J_{2}=&\frac{1}{c} \int_0^\infty \int_{S^2} \int_{\mathbb{R}^3}  \sigma_a I u_t\cdot \Omega \text{d}x \text{d}\Omega \text{d}v
\leq C|\rho|^{\frac{1}{2}}_{\infty}\int_0^\infty \int_{S^2}|\sqrt{\rho}u_t|_2 |\sigma|_{\infty}|I|_2\text{d}\Omega \text{d}v\\
\leq& \frac{1}{20}|  \sqrt{\rho} u_t|^2_{2}+C|\rho|_\infty\int_0^\infty \int_{S^2} |I |^2_{2}\text{d}\Omega \text{d}v \int_0^\infty \int_{S^2} |\sigma|^2_{\infty}\text{d}\Omega \text{d}v \\
\leq& \frac{1}{20}|  \sqrt{\rho} u_t|^2_{2} +M(c_0)c^{2}_0,\\
J_{3}= &-\frac{1}{c}\int_0^\infty \int_{S^2} \int_0^\infty \int_{S^2} \int_{\mathbb{R}^3}\frac{v}{v'}\sigma_s I'
u_t\cdot \Omega\text{d}x \text{d}\Omega' \text{d}v'\text{d}\Omega \text{d}v\\
\leq &C|\sqrt{\rho}u_t|_{2} |\rho|^{\frac{1}{2}}_\infty \int_0^\infty \int_{S^2} \int_0^\infty \int_{S^2} \frac{v}{v'}
\overline{\sigma}_s|I'|_2\text{d}\Omega' \text{d}v'\text{d}\Omega \text{d}v\\
\leq &M(c_0)\int_0^\infty \int_{S^2}  |I'|^2_2\text{d}\Omega' \text{d}v'\Big(\int_0^\infty \int_{S^2} \Big(\int_0^\infty \int_{S^2} \Big|\frac{v}{v'}\Big|^2
\overline{\sigma}^2_s \text{d}\Omega' \text{d}v'\Big)^{\frac{1}{2}}\text{d}\Omega \text{d}v\Big)^2 \\
&+\frac{1}{20}|  \sqrt{\rho} u_t|^2_{2}\leq \frac{1}{20}|  \sqrt{\rho} u_t|^2_{2}+M(c_0)c^{2}_0,\\
J_{4}= &\frac{1}{c}\int_0^\infty \int_{S^2} \int_0^\infty \int_{S^2} \int_{\mathbb{R}^3}\sigma'_s I
u_t\cdot \Omega\text{d}x \text{d}\Omega' \text{d}v'\text{d}\Omega \text{d}v\\
\leq &C|\sqrt{\rho}u_t|_{2} |\rho|^{\frac{1}{2}}_\infty \int_0^\infty \int_{S^2} \int_0^\infty \int_{S^2}
\overline{\sigma}'_s |I|_2\text{d}\Omega' \text{d}v'\text{d}\Omega \text{d}v\\
\leq &M(c_0)\int_0^\infty \int_{S^2}  |I|^2_2\text{d}\Omega \text{d}v\int_0^\infty \int_{S^2} \Big(\int_0^\infty \int_{S^2}
\overline{\sigma}'_s \text{d}\Omega' \text{d}v'\Big)^{2}\text{d}\Omega \text{d}v+\frac{1}{20}|  \sqrt{\rho} u_t|^2_{2} \\
\leq& \frac{1}{20}|  \sqrt{\rho} u_t|^2_{2}+M(c_0)c^{2}_0.
\end{split}
\end{equation*}
Combining the above estimates for $\Lambda_i$ and $J_j$, it turns out that
\begin{equation}\label{fsd1}
\begin{split}
&\frac{1}{2}\int_{\mathbb{R}^3}\rho |u_t|^2 \text{d}x+\frac{1}{2}\frac{d}{dt}\int_{\mathbb{R}^3}\Big(\mu|\nabla u|^2+\big(\lambda+\mu\big)(\text{div}u)^2\Big) \text{d}x\\
\leq &M(c_0)c^2_4|\nabla u|^2_2+M(c_0)c^2_4 -\frac{1}{c} \frac{d}{dt}\int_0^\infty \int_{S^2} \int_{\mathbb{R}^3}  S u\cdot \Omega \text{d}x \text{d}\Omega \text{d}v.
\end{split}
\end{equation}
Then integrating (\ref{fsd1}) over $(0,t)$, we have
\begin{equation*}
\begin{split}
&\int_{0}^{t}|\sqrt{\rho} u_t(s)|^2_2\text{d}s+|\nabla u(t)|^2_2\\
\leq & M(c_0)(1+c^2_4t)|\nabla u(t)|^2_2+M(c_0)c^2_4t+Cc^2_0,
\end{split}
\end{equation*}
for $0 \leq t \leq T_2$, where we have used the fact that 
$$
 \int_0^\infty \int_{S^2} \int_{\mathbb{R}^3}  S u\cdot \Omega \text{d}x \text{d}\Omega \text{d}v\leq C|u|_{\mathbb{D}^1}\int_0^\infty \int_{S^2} |S |_{\frac{6}{5}}\text{d}\Omega \text{d}v\leq Cc_0|  \nabla u|_2.
$$

From Gronwall's inequality, we have
\begin{equation}\label{plh}
\begin{split}
&\int_{0}^{t}|\sqrt{\rho} u_t(s)|^2_2\text{d}s+|\nabla u(t)|^2_2\\
\leq&  \big(M(c_0)c^2_4t+Cc^2_0\big)\exp \big(M(c_0)(1+c^2_4t)t\big)\leq Cc^{2}_0, \quad \text{for}\quad 0\leq t \leq T_2.
\end{split}
\end{equation}

According to Lemma \ref{lem:2}, (\ref{plh}) and the standard elliptic regularity estimate,  we  have
\begin{equation}\label{nv980}
\begin{split}
| u(t)|_{D^2}\leq & C\Big(|\rho u_t|_2+|\rho w\cdot\nabla u|_2+|\nabla p_m|_2+\int_0^\infty \int_{S^2}|A_r|_2 \text{d}\Omega \text{d}v\Big)(t)\\
\leq & (|\rho|^{\frac{1}{2}}_\infty|\sqrt{\rho}u_t(t)|_{2}+|\rho|_\infty|w|_6 |\nabla u(t)|_{3}+M(c_0)c_0)\\
\leq & M(c_0)\Big(|\sqrt{\rho}u_t(t)|_{2}+c_3|\nabla u|^{\frac{1}{2}}_2|\nabla^2 u|^{\frac{1}{2}}_2+c_0\Big),
\end{split}
\end{equation}
which means that 
\begin{equation}\label{nv9800}
\begin{split}
| u(t)|_{D^2}\leq &  M(c_0)\big(|\sqrt{\rho}u_t(t)|_{2}+c^2_3c_0\big).
\end{split}
\end{equation}
From (\ref{plh}) and (\ref{nv9800}), we know that 
\begin{equation}\label{plmn}
\begin{split}
\int_{0}^{t} |u|^2_{D^2}\text{d}s\leq&C \int_{0}^{t}M(c_0)\big(|\sqrt{\rho}u_t(t)|_{2}+c^2_3c_0\big)^2\text{d}s\leq M(c_0)c^{2}_0.
\end{split}
\end{equation}

$\text{\underline{Step 2}}$. The estimate of $|u|_{D^2}$.
Differentiating $ (\ref{eq:1.weeer})_3$ with respect to $t$, we have
\begin{equation}\label{da2}
\rho u_{tt}+Lu_t=-\rho_tu_t-(\rho w\cdot\nabla u)_t-(\nabla p_m)_t-\frac{1}{c}\int_0^\infty \int_{S^2}(A_r)_t\Omega \text{d}\Omega \text{d}v,
\end{equation}
multiplying  (\ref{da2}) by $u_t$ and integrating the resulting equations over $\mathbb{R}^3$, we obtain
\begin{equation*}
\begin{split}
&\frac{1}{2}\frac{\text{d}}{\text{d}t}\int_{\mathbb{R}^3}\rho |u_t|^2 \text{d}x+\int_{\mathbb{R}^3}(\mu|\nabla u_t|^2+(\lambda+\mu)(\text{div}u_t)^2) \text{d}x\\
\leq& C \int_{\mathbb{R}^3} \Big(|\rho_t w \cdot \nabla u\cdot u_t|+|\rho w_t \cdot \nabla u \cdot u_t| +|\rho w\cdot \nabla u_t \cdot u_t|
+|\ (p_m)_t||\nabla u_t|\Big)\text{d}x\\
-&\frac{1}{c}\int_{\mathbb{R}^3} \int_0^\infty \int_{S^2}  (A_r)_t u_t\cdot \Omega\text{d}\Omega \text{d}v\text{d}x=:\sum_{i=1}^{4}I_i+E_{II}.
\end{split}
\end{equation*}
First, we estimate the fluid terms $\sum_{i=1}^{4}I_i$.  According to  Lemmas \ref{lem:2}-\ref{lem:3}, Gagliardo-Nirenberg inequality,  H\"older's inequality and  Young's inequality, we easily have
\begin{equation}\label{gu1}
\begin{split}
I_1=& \int_{\mathbb{R}^3} |\rho_t w \cdot \nabla u\cdot u_t|\text{d}x \leq C|\rho_t|_{3} |w|_{\infty} |\nabla u|_{2} |  u_t|_{6}\leq M(c_0)c^6_4+\frac{\mu}{20}|  \nabla u_t|^2_{2},\\
\end{split}
\end{equation}
\begin{equation}\label{gu1gh}
\begin{split}
I_2=&  \int_{\mathbb{R}^3} |\rho w_t \cdot \nabla u \cdot u_t| \text{d}x
\leq C|\rho|^{\frac{1}{2}}_{\infty} | w_t|_{6} |\nabla u|_{2} |\sqrt{\rho}u_t|_{3}\\
\leq &\frac{1}{c^2_4} |\nabla w_t|^2_2+M(c_0)c^2_0c^2_4|\sqrt{\rho}u_t|_2|\sqrt{\rho}u_t|_6 \\
\leq& \frac{1}{c^2_4} |\nabla w_t|^2_2+M(c_0)c^8_4|\sqrt{\rho}u_t|^2_2+\frac{\mu}{20}|\nabla u_t|^2_2,\\
I_3=& \int_{\mathbb{R}^3} |\rho w\cdot \nabla u_t \cdot u_t| \text{d}x \\
\leq& C|\rho|^{\frac{1}{2}}_{\infty} |w|_{\infty} |\nabla u_t|_{2} |\sqrt{\rho}u_t|_{2}\leq M(c_0)c^2_4|\sqrt{\rho}u_t|^2_{2}+\frac{\mu}{20}|  \nabla u_t|^2_{2},\\
I_{4}=& \int_{\mathbb{R}^3} |\ (p_m)_t||\nabla u_t| \text{d}x \leq C|(p_m)_t|_{2} |\nabla u_t|_{2}\leq M(c_0)c^{2}_4+\frac{\mu}{20}|  \nabla u_t|^2_{2}.
\end{split}
\end{equation}

Now we estimate the radiation term $E_{II}$.
\begin{equation*}
\begin{split}
E_{II}=&-\frac{1}{c}\int_{\mathbb{R}^3} \int_0^\infty \int_{S^2}  (A_r)_t u_t\cdot \Omega\text{d}\Omega \text{d}v\text{d}x \\
=&-\frac{1}{c}\int_0^\infty \int_{S^2} \int_{\mathbb{R}^3} \Big(S_t-(\sigma_a)_t I-\sigma_a I_t+ \int_0^\infty \int_{S^2} \frac{v}{v'}(\sigma_s I')_t\text{d}\Omega'\text{d}v'\Big)u_t \cdot \Omega\text{d}x \text{d}\Omega \text{d}v\\
&+\frac{1}{c}\int_0^\infty \int_{S^2}\int_0^\infty \int_{S^2} \int_{\mathbb{R}^3} (\sigma'_s I)_t u_t\cdot \Omega\text{d}x \text{d}\Omega'\text{d}v' \text{d}\Omega \text{d}v
=:\sum_{j=5}^{9} J_j.
\end{split}
\end{equation*}
We estimate $J_{j}$ term by term. From  Lemmas \ref{lem:2}-\ref{lem:3}, Gagliardo-Nirenberg inequality, H\"older's inequality, Young's inequality and  (\ref{zhen1})-(\ref{jia345}), for $0\leq t\leq T_2$ we have
\begin{equation*}
\begin{split}
J_5= &-\frac{1}{c} \int_0^\infty \int_{S^2} \int_{\mathbb{R}^3}  S_t u_t\cdot \Omega \text{d}x \text{d}\Omega \text{d}v\\
\leq& C|u_t|_{\mathbb{D}^1}\int_0^\infty \int_{S^2} |S_t |_{\frac{6}{5}}\text{d}\Omega \text{d}v
\leq \frac{\mu}{20}|  \nabla u_t|^2_{2}+Cc^2_0,\\
J_{6}=&\frac{1}{c} \int_0^\infty \int_{S^2} \int_{\mathbb{R}^3}  (\sigma_a)_t I u_t\cdot \Omega \text{d}x \text{d}\Omega \text{d}v
\leq C|u_t|_{6}\int_0^\infty \int_{S^2} |(\sigma_a)_t |_{2}\|I\|_1\text{d}\Omega \text{d}v\\
\leq& \frac{\mu}{20}|  \nabla u_t|^2_{2}+C\int_0^\infty \int_{S^2} \|I \|^2_{1}\text{d}\Omega \text{d}v \int_0^\infty \int_{S^2} \big(|\sigma_t|^2_{2}|\rho|^2_\infty+|\sigma|^2_{\infty}|\rho_t|^2_2\big)\text{d}\Omega \text{d}v \\
\leq& \frac{\mu}{20}|  \nabla u_t|^2_{2} +M(c_0)c^{4}_4 \int_0^\infty \int_{S^2} \big(|\sigma_t|^2_{2}+|\sigma|^2_{\infty}\big)\text{d}\Omega \text{d}v \\
\leq& \frac{\mu}{20}|  \nabla u_t|^2_{2} + M(c_0)c^{4}_4 (|\rho_t |^2_{2}|\rho|^2_\infty+|\rho|^2_\infty) \leq \frac{\mu}{20}|  \nabla u_t|^2_{2} +M(c_0)c^{6}_4,\\
\end{split}
\end{equation*}
where we used the fact that
\begin{equation*}
\begin{split}
&\int_0^\infty \int_{S^2} |(\sigma_a)_t |^2_{2}\text{d}\Omega \text{d}v\leq \int_0^\infty \int_{S^2}\big(|\sigma_t|^2_{2}|\rho|^2_\infty+|\sigma|^2_{\infty}|\rho_t|^2_2\big)\text{d}\Omega \text{d}v\leq M(|\rho|_\infty)|\rho|^2_\infty (|\rho_t |^2_{2}+1).
\end{split}
\end{equation*}
And similarly
\begin{equation*}
\begin{split}
J_{7}=& \frac{1}{c}\int_0^\infty \int_{S^2} \int_{\mathbb{R}^3}  \sigma_a I_t u_t\cdot \Omega \text{d}x \text{d}\Omega \text{d}v
\leq C|\rho|^{\frac{1}{2}}_\infty|\sqrt{\rho}u_t|_{2}\int_0^\infty \int_{S^2} |\sigma |_{\infty}|I_t|_2\text{d}\Omega \text{d}v\\
\leq&  C|\sqrt{\rho}u_t|^2_{2}+M(c_0) \int_0^\infty \int_{S^2} |I_t|^2_2\text{d}\Omega \text{d}v\int_0^\infty \int_{S^2} |\sigma |^2_{\infty}\text{d}\Omega \text{d}v\\
\leq&  C|\sqrt{\rho}u_t|^2_{2}+M(c_0)c^{2}_4,\\
J_{8}= &-\frac{1}{c}\int_0^\infty \int_{S^2} \int_0^\infty \int_{S^2} \int_{\mathbb{R}^3}\frac{v}{v'}(\sigma_s I')_t
u_t\cdot \Omega\text{d}x \text{d}\Omega' \text{d}v'\text{d}\Omega \text{d}v\qquad\qquad\qquad\quad\\
\leq &C|u_t|_{\mathbb{D}^1} |\rho_t|_2 \int_0^\infty \int_{S^2} \int_0^\infty \int_{S^2} \frac{v}{v'}
\overline{\sigma}_s\|I'\|_1\text{d}\Omega' \text{d}v'\text{d}\Omega \text{d}v\qquad\qquad\qquad\quad\\
&+C|\rho|^{\frac{1}{2}}_\infty|\sqrt{\rho}u_t|_{2}\int_0^\infty \int_{S^2} \int_0^\infty \int_{S^2} \frac{v}{v'}
\overline{\sigma}_s |I'_t|_2\text{d}\Omega' \text{d}v'\text{d}\Omega \text{d}v\\
\end{split}
\end{equation*}
\begin{equation*}
\begin{split}
\leq &M(c_0)c^2_4\int_0^\infty \int_{S^2}  \|I'\|^2_1\text{d}\Omega' \text{d}v'\Big(\int_0^\infty \int_{S^2} \Big(\int_0^\infty \int_{S^2} \Big|\frac{v}{v'}\Big|^2
\overline{\sigma}^2_s \text{d}\Omega' \text{d}v'\Big)^{\frac{1}{2}}\text{d}\Omega \text{d}v\Big)^2 \\
& + M(c_0)\int_0^\infty \int_{S^2} |I'_t|^2_2\text{d}\Omega' \text{d}v'\Big(\int_0^\infty \int_{S^2} \Big(\int_0^\infty \int_{S^2} \Big|\frac{v}{v'}\Big|^2
\overline{\sigma}^2_s \text{d}\Omega' \text{d}v'\Big)^{\frac{1}{2}}\text{d}\Omega \text{d}v\Big)^2 \\
&+\frac{\mu}{20}|u_t|^2_{\mathbb{D}^1}+C|\sqrt{\rho}u_t|^2_{2}\leq \frac{\mu}{20}|u_t|^2_{\mathbb{D}^1}+C|\sqrt{\rho}u_t|^2_{2}+M(c_0)c^{4}_4,\\
J_{9}= &\frac{1}{c}\int_0^\infty \int_{S^2} \int_0^\infty \int_{S^2} \int_{\mathbb{R}^3}(\sigma'_s I)_t
u_t\cdot \Omega\text{d}x \text{d}\Omega' \text{d}v'\text{d}\Omega \text{d}v\\
\leq &C|u_t|_{\mathbb{D}^1}|\rho_t|_2 \int_0^\infty \int_{S^2} \int_0^\infty \int_{S^2}
\overline{\sigma}'_s \|I\|_1\text{d}\Omega' \text{d}v'\text{d}\Omega \text{d}v\\
&+C|\rho|^{\frac{1}{2}}_\infty|\sqrt{\rho}u_t|_{2}\int_0^\infty \int_{S^2} \int_0^\infty \int_{S^2}
\overline{\sigma}'_s |I_t|_2\text{d}\Omega' \text{d}v'\text{d}\Omega \text{d}v\\
\leq &M(c_0)c^2_4\int_0^\infty \int_{S^2}  \|I\|^2_1\text{d}\Omega \text{d}v\int_0^\infty \int_{S^2} \Big(\int_0^\infty \int_{S^2}
\overline{\sigma}'_s \text{d}\Omega' \text{d}v'\Big)^{2}\text{d}\Omega \text{d}v+\frac{\mu}{20}|u_t|^2_{\mathbb{D}^1} \\
&+M(c_0)\int_0^\infty \int_{S^2} |I_t|^2_2\text{d}\Omega \text{d}v\int_0^\infty \int_{S^2}\Big( \int_0^\infty \int_{S^2}
\overline{\sigma}'_s \text{d}\Omega' \text{d}v'\Big)^{2}\text{d}\Omega \text{d}v+C|\sqrt{\rho}u_t|^2_{2} \\
\leq& \frac{\mu}{20}|u_t|^2_{\mathbb{D}^1}+C|\sqrt{\rho}u_t|^2_{2}+M(c_0)c^{4}_4.
\end{split}
\end{equation*}
Combining the above estimates for $I_i$ and $J_j$, it turns out that
\begin{equation}\label{kaka}
\begin{split}
&\frac{1}{2}\frac{\text{d}}{\text{d}t}\int_{\mathbb{R}^3}\rho |u_t|^2 \text{d}x+\frac{1}{3}\int_{\mathbb{R}^3}(\mu|\nabla u_t|^2+(\lambda+\mu)(\text{div}u_t)^2) \text{d}x\\
\leq& M(c_0)c^{6}_4+\frac{1}{c^2_4} |\nabla w_t|^2_2+M(c_0)c^8_4|\sqrt{\rho}u_t|^2_{2}+M(c_0)c_0| u|_{D^2}.
\end{split}
\end{equation}
Integrating (\ref{kaka}) over $(\tau,t)$ for $\tau\in (0,t)$, we easily get
 \begin{equation}\label{nv4}
\begin{split}
&|\sqrt{\rho}u_t(t)|^2_{2}+\int_{\tau}^{t}|u_t(s)|^2_{\mathbb{D}^1}\text{d}s\\
\leq& |\sqrt{\rho}u_t(\tau)|^2_{2}
+\int_{\tau}^{t}\Big(M(c_0)c^8_4|\sqrt{\rho}u_t|^2_{2}+M(c_0)c_0| u|_{D^2}\Big)\text{d}s+M(c_0)c^{6}_4t+C.
\end{split}
\end{equation}
From the momentum equations $(\ref{eq:1.weeer})_3$, we have
\begin{equation}\label{nv5}
\begin{split}
&|\sqrt{\rho}u_t(\tau)|^2_{2}\leq C|\rho(\tau)|_{\infty} \|\nabla w(\tau) \|^2_1 |\nabla u(\tau)|^2_2+C\int_{\mathbb{R}^3}\frac{|\Phi(\tau)|^2}{\rho(\tau)}\text{d}x,
\end{split}
\end{equation}
where
$$
\Phi(\tau)=\nabla p_m(\tau) +Lu(\tau)+\frac{1}{c}\int_0^\infty \int_{S^2}A_r(\tau)\Omega \text{d}\Omega \text{d}v.
$$
From the assumptions (\ref{zhen1})-(\ref{jia345}),  Lemma \ref{lem1},  the regularity of $S(v,\Omega, t,x)$ and Minkowski inequality, we easily have
\begin{equation*}
\begin{split}
&\lim_{\tau \mapsto 0}\int_{\mathbb{R}^3}\Big(\frac{|\Phi(\tau)|^2}{\rho(\tau)}-\frac{|\Phi(0)|^2}{\rho_0}\Big)\text{d}x\\
\leq&
\lim_{\tau \mapsto 0}\Big(\frac{1}{\underline{\delta}}\int_{\mathbb{R}^3}|\Phi(\tau)-\Phi(0)|^2\text{d}x+\frac{1}{\delta\underline{\delta}}|\rho(\tau)-\rho_0|_\infty\int_{\mathbb{R}^3}|\Phi(0)|^2\text{d}x\Big)
=0.
\end{split}
\end{equation*}
According to the compatibility condition (\ref{kkkkk}), it is easy to show that
\begin{equation}\label{nv6}
\begin{split}
\limsup_{\tau \rightarrow 0}|\sqrt{\rho}u_t(\tau)|^2_{2}\leq C|\rho_0|_{\infty} \|\nabla u_0 \|^2_1 |\nabla u_0|^2_2+C|g_1|^2_2\leq Cc^5_0.
\end{split}
\end{equation}
Therefore, by letting $\tau\rightarrow 0$ in (\ref{nv4}) and (\ref{plmn}), we have
 \begin{equation}\label{ghj}
\begin{split}
&|\sqrt{\rho}u_t(t)|^2_{2}+\int_{0}^{t}|u_t(s)|^2_{\mathbb{D}^1}\text{d}s\\
\leq&\int_{0}^{t}M(c_0)c^8_4|\sqrt{\rho}u_t|^2_{2}\text{d}s+M(c_0)c^{6}_4t+M(c_0)c^5_0,
\end{split}
\end{equation}
for $0\leq t\leq T_2$.

From Gronwall's inequality, we get
\begin{equation}\label{nv98nn}
\begin{split}
&|\sqrt{\rho}u_t(t)|^2_{2}+\mu\int_{0}^{t}| u_t(s)|^2_{\mathbb{D}^1}\text{d}s\\
\leq& (M(c_0)c^{6}_4t+M(c_0)c^5_0) \exp\big(M(c_0)c^8_4t\big)\leq M(c_0)c^{5}_0
\end{split}
\end{equation}
for $0\leq t \leq  T_3= \min(T^*,(1+M(c_0)c^{8}_4)^{-1})$. Combining (\ref{nv9800}) and (\ref{nv98nn}) yields
\begin{equation*}
\begin{split}
| u(t)|_{D^2}
\leq    M(c_0)\big(|\sqrt{\rho}u_t(t)|_{2}+c^2_3c_0\big)\leq M(c_0)c^{3}_3.
\end{split}
\end{equation*}

Finally, from the standard  elliptic regularity estimate (see \cite{CK2}) and Minkowski inequality, we conclude that
\begin{equation*}
\begin{split}
&\int_{0}^{t}|u(s)|^2_{D^{2,q}}\text{d}s
\leq \int_{0}^{t}\Big( |\rho u_t|^2_q+|\rho w\cdot\nabla u|^2_q+|\nabla p_m|^2_q\Big)(s)\text{d}s\\
&\qquad \qquad \qquad+\int_{0}^{t} \Big(\int_0^\infty \int_{S^2}|A_r|_q\text{d}\Omega \text{d}v\Big)^2(s)\text{d}s\leq M(c_0)c^{10}_3
\end{split}
\end{equation*}
for  $0\leq t \leq T_3$.
\end{proof}

Based on  Lemmas \ref{lem:2}-\ref{lem:4}, we obtain the following local (in time) a priori estimate independent of the lower bound $\delta$ of the initial mass density $\rho_0$:
\begin{equation}\label{priaa}\begin{split}
\|\rho(t)-\overline{\rho}\|_{H^1\cap W^{1,q}}+\|p_m(t)-\overline{p}\|_{H^1\cap W^{1,q}}\leq& M(c_0),\\
 |\rho_t(t)|_{2}+ |(p_m)_t(t)|_{2}\leq M(c_0)c_3,\quad  |\rho_t(t)|_{q}+ |(p_m)_t(t)|_{q}\leq& M(c_0)c_4,\\
\displaystyle
\|I\|_{L^2(\mathbb{R}^+\times S^2;C([0,T_*];H^1\cap W^{1,q}(\mathbb{R}^3)))}\leq& M(c_0)c_0\\
 \|I_t\|_{L^2(\mathbb{R}^+\times S^2;C([0,T_*];L^{2}\cap L^{q}(\mathbb{R}^3)))}\leq& M(c_0)c_0,\\
\displaystyle
|u(t)|^2_{\mathbb{D}^1}+\int_{0}^{T_*}\Big(|u|^2_{D^2}+|\sqrt{\rho}u_t|^2_{2}\Big)(t)\text{d}t\leq& M(c_0)c^{2}_0,\\
\displaystyle
\big(|u(t)|^2_{\mathbb{D}^1\cap D^2}+|\sqrt{\rho}u_t(t)|^2_{2}\big)+\int_{0}^{T_*}\Big(|u|^2_{D^{2,q}}+|u_t|^2_{\mathbb{D}^1}\Big)(t)\text{d}t\leq& M(c_0)c^{10}_3,
\end{split}
\end{equation}
for  $0\leq t\leq T_3$. Therefore, if we define the constants $c_i$ ($i=1,2,3,4$) and $T^*$ by
\begin{equation}\label{dingyi}
\begin{split}
&c_1=c_2=c_3=M(c_0)c_0,\\
& c_4= M(c_0)c^5_3=M^6(c_0)c^5_0, \quad \text{and} \quad T^*=(T, (1+M(c_0)c^8_4)^{-1}),
\end{split}
\end{equation}
then we deduce that 
\begin{equation}\label{pri}\begin{split}
\displaystyle
\|I\|_{L^2(\mathbb{R}^+\times S^2;C([0,T^*];H^1\cap W^{1,q}(\mathbb{R}^3)))}\leq& c_1,\\
 \|I_t\|_{L^2(\mathbb{R}^+\times S^2;C([0,T^*];L^{2}\cap L^{q}(\mathbb{R}^3)))}\leq& c_2,\\
\displaystyle
\sup_{0\leq t\leq T^*}|u(t)|^2_{\mathbb{D}^1}+\int_{0}^{T^*}\Big(|u|^2_{D^2}+|\sqrt{\rho}u_t|^2_{2}\Big)(t)\text{d}t\leq& c^2_3,\\
\displaystyle
\text{ess}\sup_{0\leq t\leq T^*}\big(|u(t)|^2_{\mathbb{D}^1\cap D^2}+|\sqrt{\rho}u_t(t)|^2_{2}\big)+\int_{0}^{T^*}\Big(|u|^2_{D^{2,q}}+|u_t|^2_{\mathbb{D}^1}\Big)(t)\text{d}t\leq& c^2_4,\\
\sup_{0\leq t\leq T^*}(\|\rho(t)-\overline{\rho}\|_{H^1\cap W^{1,q}}+\|p_m(t)-\overline{p}\|_{H^1\cap W^{1,q}})\leq& c_1,\\
\sup_{0\leq t\leq T^*} (|\rho_t(t)|_{2}+ |(p_m)_t(t)|_{2}+|\rho_t(t)|_{q}+ |(p_m)_t(t)|_{q})\leq& c^2_4.
\end{split}
\end{equation}

\subsection{The unique solvability of  the linearized problem with vacuum}\ \\

First we give the following key lemma for the proof of our main result - Theorem \ref{th1}.
\begin{lemma}\label{lemk1}
Let (\ref{tyuu}) hold. 
If $(I_0,\rho_0, u_0)$ satisfies the compatibility condition
\begin{equation}\label{kk2k}
\begin{split}
Lu_0+\nabla p^0_{m}+\frac{1}{c}\int_0^\infty \int_{S^2}A^0_r\Omega \text{d}\Omega \text{d}v=&\rho^{\frac{1}{2}}_0 g_1
\end{split}
\end{equation}
for some $g_1\in L^2$, where $p^0_m=A\rho^\gamma_0,\  A^0_r=A_r(v,\Omega,t=0,x,\rho_0,I_0,I'_0),$ then there exists a unique strong solution $(I,\rho,u)$ to the Cauchy problem (\ref{eq:1.weeer}) with (\ref{eq:2.2hh}) such that
\begin{equation*}\begin{split}
&I\in L^2(\mathbb{R}^+\times S^2;C([0,T^*];H^1\cap W^{1,q}(\mathbb{R}^3))), \ I_t\in L^2(\mathbb{R}^+\times S^2;C([0,T^*];L^{2}\cap L^{q}(\mathbb{R}^3))),\\
&\rho\geq  0,\ \ \rho-\overline{\rho}\in C([0,T^*];H^1\cap W^{1,q}),\quad \rho_t\in C([0,T^*];L^{2}\cap L^{q}),\\
&u\in C([0,T^*];\mathbb{D}^1\cap D^2)\cap  L^2([0,T^*];D^{2,q}),\ u_{t}\in  L^2([0,T^*];\mathbb{D}^1),\ \sqrt{\rho}u_{t}\in  L^\infty([0,T^*];L^2).
\end{split}
\end{equation*}
Moreover, $(I,\rho,u)$ satisfies the local estimate (\ref{pri}).
\end{lemma}

\begin{proof}We divide the proof into three steps.\\
\underline{Step 1}: Existence.
Let $\delta> 0$ be a constant, and for each $\delta \in (0,1)$, define
\begin{equation*}
\begin{split}
&\rho_{\delta0}=\rho_0+\delta,\,\, (p_m)_{\delta0}=A(\rho_0+\delta)^\gamma,\\
&A^{\delta0}_r=A_r(v,\Omega,t=0,x,\rho_{\delta0},I_0,I'_0).
\end{split}
\end{equation*}
Then from the compatibility condition (\ref{kkkkk}) we have
\begin{equation*}
\begin{split}
&Lu_0+A\nabla \rho^\gamma_{\delta0}+\frac{1}{c}\int_0^\infty \int_{S^2}A^{\delta0}_r\Omega \text{d}\Omega \text{d}v=(\rho_{\delta0})^{\frac{1}{2}} g^\delta_1,
\end{split}
\end{equation*}
where
\begin{equation*}\begin{split}
g^\delta_1=&\Big(\frac{\rho_0}{\rho_{\delta0}}\Big)^{\frac{1}{2}}g_1+A\frac{  \nabla(\rho^\gamma_{\delta0}-\rho^\gamma_0)}{(\rho_{\delta0})^{\frac{1}{2}}}-\frac{1}{c}\int_0^\infty \int_{S^{2}}\frac{(A^{0}_r-A^{\delta0}_r)}{(\rho_{\delta0})^{\frac{1}{2}}}\Omega\text{d}\Omega \text{d}v.
\end{split}
\end{equation*}

It is easy to know from the assumptions (\ref{zhen1})-(\ref{jia345}), that for all small $\delta> 0$,
\begin{equation*}\begin{split}
&1+\|\rho_{\delta0}-\overline{\rho}-\delta\|_{H^1\cap W^{1,q}}+|u_0|_{\mathbb{D}^1\cap D^2}+\|I_0\|_{L^2(\mathbb{R}^+\times S^2; H^1\cap W^{1,q}(\mathbb{R}^3))}+|g^\delta_1|_2\\
&+\|S(v,\Omega, t,x)\|_{ L^2(\mathbb{R}^+\times S^2;C^1([0,T^*];H^1\cap W^{1,q}(\mathbb{R}^3)))\cap C^1([0,T^*]; L^1(\mathbb{R}^+\times S^2;L^1\cap L^2(\mathbb{R}^3)))}\leq c_0.
\end{split}
\end{equation*}
Therefore, corresponding to initial data $(I_0,\rho_{\delta0},u_0)$, there exists a unique strong solution $(I^\delta,\rho^\delta, u^\delta)$ satisfying the local estimate (\ref{pri}). Thus we can choose a subsequence of solutions (still denoted by $(I^\delta,\rho^\delta,u^\delta)$) converging to a limit $(I,\rho,u)$ in weak or weak* sense. Furthermore, for any $R> 0$, thanks to the compact property \cite{jm}, there exists a subsequence (still denoted by $(I^\delta,\rho^\delta,u^\delta)$) satisfying
\begin{equation}\label{ert}\begin{split}
& I^\delta\rightarrow  I\ \text{weakly}\  \text{in}\  L^2(\mathbb{R}^+\times S^2;C([0,T^*];H^1\cap W^{1,q}(\mathbb{R}^3))),\\
&\rho^\delta\rightarrow \rho\ \text{in }\ C([0,T^*];L^2(B_R)),\\
&u^\delta\rightarrow u \ \text{in } \ C([0,T^*];H^1(B_R)),
\end{split}
\end{equation}
where $B_{R}=\{x \in \mathbb{R}^3: |x|<R\}$. By the lower semi-continuity of norms (see \cite{gandi}), it follows from (\ref{ert}) that $(I,\rho,u)$ also satisfies the  estimate (\ref{pri}).
For any $\varphi\in C^\infty_c(\mathbb{R}^+\times S^2\times [0,T^*]\times \mathbb{R}^3)$, from  (\ref{pri}), (\ref{ert}) and assumptions (\ref{zhen1})-(\ref{jia345}) we easily have
\begin{equation*}\begin{split}
\int_0^\infty \int_{S^2}\int_0^{T^*} \int_{\mathbb{R}^3} (\Lambda(v,\Omega,t,x,\rho^\delta)I^\delta -\Lambda(v,\Omega,t,x,\rho)I)\varphi \text{d}x \text{d}t\text{d}\Omega \text{d}v\rightarrow 0,
\ \text{as} \ \delta\rightarrow 0,
\end{split}
\end{equation*}
where $\displaystyle \Lambda=\sigma_a+\int_{0}^{\infty}\int_{S^{2}} \sigma'_s\text{d}\Omega' \text{d}v'$,
and
\begin{equation*}\begin{split}
\int_0^\infty \int_{S^{2}}\int_{0}^{T^*}\int_{\mathbb{R}^3}\int_0^\infty \int_{S^{2}}\frac{v}{v'}\Big(\sigma_s(\rho^\delta)I'^\delta-\sigma_s(\rho)I'\Big)\xi
\text{d}\Omega' \text{d}v'\text{d}x \text{d}t\text{d}\Omega \text{d}v\rightarrow 0,
\ \text{as} \ \delta\rightarrow 0.
\end{split}
\end{equation*}
Then it is easy to show that $(I,\rho,u)$ is a weak solution in the sense of distribution and satisfies the following regularities:
\begin{equation*}\begin{split}
&I\in L^2(\mathbb{R}^+\times S^2;L^\infty([0,T^*];H^1\cap W^{1,q}(\mathbb{R}^3))),\ I_t\in L^2(\mathbb{R}^+\times S^2;L^\infty([0,T^*];L^{2}\cap L^{q}(\mathbb{R}^3))),\\
&\rho\geq 0,\quad \rho-\overline{\rho}\in L^\infty([0,T^*];H^1\cap W^{1,q})),\quad \rho_t\in L^\infty([0,T^*];L^{2}\cap L^{q}),\\
&u\in L^\infty([0,T^*];\mathbb{D}^1\cap D^2)\cap  L^2([0,T^*];D^{2,q}),\ u_{t}\in  L^2([0,T^*];\mathbb{D}^1),\ \sqrt{\rho}u_{t}\in  L^\infty([0,T^*];L^2).
\end{split}
\end{equation*}
\underline{Step 2}: Uniqueness. Let $(I_1,\rho_1,u_1)$ and $(I_2,\rho_2,u_2)$ be two solutions obtained in step 1 with the same initial data. Then by the same method as in \cite{CK} we can get$\rho_1=\rho_2$  and $u_{1}=u_{2}$. Here we omit the details. It is easy to show that $I_1-I_2$ satisfies the following Cauchy problem:
 $$\frac{1}{c}(I_1-I_2)_t+\Omega\cdot\nabla (I_1-I_2)+\Lambda (I_1-I_2)=0, \quad (I_1-I_2)|_{t=0}=0.$$
It follows immediately that $I_{1}=I_{2}$. \\
\underline{Step 3}: The time-continuity. The continuity of $\rho$ can be obtained analogously to \cite{CK}. As for $I$,
by Lemma \ref{lem:2}, for $\forall$ $(v,\Omega)\in R^+\times S^2$ we have
$$
I(v,\Omega,\cdot,\cdot)\in C([0,T^*]; L^2\cap L^q(\mathbb{R}^3)) \cap C([0,T^*]; W^{1,2}\cap W^{1,q}(\mathbb{R}^3)-\text{weak}).
$$
According to ({\ref{coti}}), we have
\begin{equation}\label{thymm}
\begin{split}
\limsup_{t\rightarrow 0}\|I(v,\Omega,\cdot,\cdot)\|^2_{W^{1,r}}\leq \|I_0\|^2_{W^{1,r}},
\end{split}
\end{equation}
which implies that $I(v,\Omega,t,x)$ is right-continuous at $t=0$ (see \cite{teman}). Similarly, form Lemmas \ref{lem:3}-\ref{lem:4} we have
$$
u\in C([0,T^*]; \mathbb{D}^1) \cap C([0,T^*]; D^2-\text{weak}).
$$
From equations $(\ref{eq:1.weeer})_3$ and Lemmas \ref{lem:2}-\ref{lem:4}, we also know that
$$
 \rho u_t\in L^2([0,T^*]; L^2), \ \text{and}\  (\rho u_t)_t\in L^2([0,T^*]; H^{-1}).
$$
From Aubin-Lions lemma we then have  $ \rho u_t \in C([0,T^*]; L^2)$.
From
\begin{equation}\label{thghj}
\displaystyle
Lu=-\rho u_t-\rho w\cdot \nabla u
  -\nabla p_m +\frac{1}{c}\int_0^\infty \int_{S^2}A_r\Omega \text{d}\Omega \text{d}v,
\end{equation}
and the standard elliptic regularity estimates (see \cite{CK2}), we get $u\in C([0,T^*]; D^2)$.
\end{proof}

\subsection{Proof of Theorem \ref{th1}}\ \\

Our proof  is based on the classical iteration scheme and the existence results for the linearized problem obtained  in  Section $3.2$. Like in Section 3.2, we define constants $c_{0}$ and  $c_{1}$, $c_2$, $c_3$, $c_4$ and assume that
\begin{equation*}\begin{split}
&2+\|\rho_0-\overline{\rho}\|_{H^1\cap W^{1,q}}+|u_0|_{\mathbb{D}^1\cap D^2}+\|I_0\|_{L^2(\mathbb{R}^+\times S^2; H^1\cap W^{1,q}(\mathbb{R}^3))}+|g_1|_2\\
&+\|S(v,\Omega, t,x)\|_{ L^2(\mathbb{R}^+\times S^2;C^1([0,T^*];H^1\cap W^{1,q}(\mathbb{R}^3)))\cap C^1([0,T^*]; L^1(\mathbb{R}^+\times S^2;L^1\cap L^2(\mathbb{R}^3)))}\leq c_0.
\end{split}
\end{equation*}
Let $u^0\in C([0,T^*];\mathbb{D}^1\cap D^2)\cap  L^2([0,T^*];D^{2,q}) $ be the solution to the linear parabolic problem
$$
h_t-\triangle h=0 \quad \text{in} \quad (0,+\infty)\times \mathbb{R}^3 \quad \text{and} \quad h|_{t=0}=u_0 \quad \text{in} \quad \mathbb{R}^3.
$$
Let $I^0\in L^2(\mathbb{R}^+\times S^2;C([0,T^*];H^1\cap W^{1,q}(\mathbb{R}^3)))$ be the solution to  the linear parabolic problem
$$
f_t+c\Omega\cdot\nabla f=0, \ \text{in} \ \mathbb{R}^+\times S^2\times (0,+\infty)\times \mathbb{R}^3 \quad \text{and} \quad f|_{t=0}=I_0 \quad \text{in} \quad \mathbb{R}^+\times S^2\times \mathbb{R}^3.
$$
Taking a small time $\overline{T}_1\in (0,T^*)$, we then have
\begin{equation*}\begin{split}
\|I^0\|_{L^2(\mathbb{R}^+\times S^2;C([0,\overline{T}_1];H^1\cap W^{1,q}(\mathbb{R}^3)))}\leq& c_1,\\
 \|I^0_t\|_{L^2(\mathbb{R}^+\times S^2;C([0,\overline{T}_1];L^{2}\cap L^{q}(\mathbb{R}^3)))}\leq& c_2,\\
\sup_{0\leq t\leq \overline{T}_1 }|u^0(t)|^2_{\mathbb{D}^1}+\int_{0}^{\overline{T}_1}|u^0(t)|^2_{D^2}\text{d}t\leq& c^2_3,\\
\sup_{0\leq t\leq \overline{T}_1 }|u^0(t)|^2_{ D^2}+\int_{0}^{\overline{T}_1}\big(|u^0(t)|^2_{D^{2,q}}+|u^0_t(t)|^2_{\mathbb{D}^1}\big)\text{d}t\leq& c^2_4.
\end{split}
\end{equation*}
We divided the proof of Theorem \ref{th1} into two cases: $\overline{\rho}>0$ and $\overline{\rho}=0$.
\subsubsection{\bf{ Case $\overline{\rho}>0$.}}\ \\
\begin{proof} The proof of this case is divided into three steps.\\[2mm]
\underline{Step 1}. The existence of strong solutions. Let $(I^1, \rho^1, u^1)$ be the strong solution to Cauchy problem (\ref{eq:1.weeer}) with (\ref{eq:2.2hh}) and $(w,\psi)=(u^0,I'^0)$. Then we construct approximate solutions $(I^{k+1}, \rho^{k+1}, u^{k+1})$ inductively as follows. Assume that $(I^{k},\rho^k, u^{k})$ was defined for $k\geq 1$, let $(I^{k+1}, \rho^{k+1}, u^{k+1})$  be the unique solution to the Cauchy problem  (\ref{eq:1.weeer}) with (\ref{eq:2.2hh}) with $(w,\psi)$=$( u^{k},I'^{k})$:
\begin{equation}
\label{eq:1.wgo}
\begin{cases}
\displaystyle
\rho^{k+1}_t+\text{div}(\rho^{k+1} u^k)=0,\\[5pt]
\displaystyle
\frac{1}{c}I^{k+1}_t+\Omega\cdot\nabla I^{k+1}=\overline{A}^{k}_r,\\[5pt]
\displaystyle
\rho^{k+1} u^{k+1}_t+\rho^{k+1} u^{k}\cdot \nabla u^{k+1}
  +\nabla p^{k+1}_m +Lu^{k+1}=-\frac{1}{c}\int_0^\infty \int_{S^2}A^{k}_r\Omega \text{d}\Omega \text{d}v
\end{cases}
\end{equation}
with initial data
$$
(I^{k+1}, \rho^{k+1}, u^{k+1})|_{t=0}=(I_0, \rho_0, u_0),
$$
where
\begin{equation*}
\begin{split}
&\overline{A}^{k}_r=S-\sigma^{k+1}_aI^{k+1}+\int_0^\infty \int_{S^2}\Big( \frac{v}{v'}\sigma^{k+1}_sI'^{k} -(\sigma'_s)^{k+1}I^{k+1} \Big)\text{d}\Omega' \text{d}v',\\
&A^{k}_r=S-\sigma^{k+1}_aI^{k+1}+\int_0^\infty \int_{S^2}\Big( \frac{v}{v'}\sigma^{k+1}_sI'^{k+1} -(\sigma'_s)^{k+1}I^{k+1} \Big)\text{d}\Omega' \text{d}v',\\
&p^{k+1}_m=A(\rho^{k+1})^\gamma,\quad
\sigma^{k+1}_a=\sigma_a(v,\Omega,t,x,\rho^{k+1}),\\
& \sigma^{k+1}_s=\sigma_s(v'\rightarrow v,\Omega'\cdot\Omega,\rho^{k+1}),\ (\sigma'_s)^{k+1}=\sigma_s(v\rightarrow v',\Omega\cdot\Omega',\rho^{k+1}).
\end{split}
\end{equation*}
According to the arguments in Sections $3.1$-$3.2$,  we know that the solution sequence $(I^k,\rho^k,p^k_m,u^k)$ still satisfies the priori estimates (\ref{pri}).

Now we show that $(I^k, \rho^k, u^k)$ converges to a limit in a strong sense. Let
$$
\overline{I}^{k+1}=I^{k+1}-I^k,\ \overline{\rho}^{k+1}=\rho^{k+1}-\rho^k,\ \overline{p}^{k+1}_m=p^{k+1}_m-p^k_m,\ \overline{u}^{k+1}=u^{k+1}-u^k,
$$
then we have
\begin{equation}
\label{eq:1.2w}
\begin{cases}
\displaystyle
\overline{\rho}^{k+1}_t+\text{div}(\overline{\rho}^{k+1} u^k)+\text{div}(\rho^{k} \overline{u}^k)=0,\\[8pt]
\displaystyle
\frac{1}{c}\overline{I}^{k+1}_t+\Omega\cdot\nabla \overline{I}^{k+1}+\Big(\sigma^{k+1}_a+\int_0^\infty \int_{S^2}(\sigma'_s)^{k+1}\text{d}\Omega' \text{d}v'\Big)\overline{I}^{k+1}
=L_1,\\[8pt]
\displaystyle
\rho^{k+1} \overline{u}^{k+1}_t+\rho^{k+1} u^k\cdot\nabla \overline{u}^{k+1}+L\overline{u}^{k+1}\\[8pt]
=\overline{\rho}^{k+1}(-u^k_t-u^{k-1}\cdot\nabla u^k)
-\rho^{k+1}\overline{u}^{k}\cdot\nabla u^k-\nabla \overline{p}^{k+1}_m+L_2,
\end{cases}
\end{equation}
where $L_1$ and  $L_2$ are given by
\begin{equation*}\begin{split}
L_1=&-I^k(\sigma^{k+1}_a-\sigma^{k}_a)-\int_0^\infty \int_{S^2}
\Big((\sigma'_s)^{k+1}-(\sigma'_s)^{k}\Big)I^{k}\text{d}\Omega' \text{d}v'\\
&+\int_0^\infty \int_{S^2} \Big(\frac{v}{v'}
\big(\sigma^{k}_s\overline{I}'^{k}+I'^k(\sigma^{k+1}_s-\sigma^{k}_s)\big)\Big)\text{d}\Omega' \text{d}v',\\[6pt]
L_2=&-\frac{1}{c}\int_0^\infty \int_{S^2}\Omega\Big(-\sigma^{k+1}_a\overline{I}^{k+1}-I^k(\sigma^{k+1}_a-\sigma^{k}_a)\Big) \text{d}\Omega \text{d}v\\
&-\frac{1}{c}\int_0^\infty \int_{S^2}\int_0^\infty \int_{S^2}\Omega \frac{v}{v'}\Big(\sigma^{k+1}_s\overline{I}'^{k+1}+I'^k(\sigma^{k+1}_s-\sigma^{k}_s)\Big)\text{d}\Omega' \text{d}v'\text{d}\Omega \text{d}v\\
&-\frac{1}{c}\int_0^\infty \int_{S^2} \int_0^\infty \int_{S^2}-\Omega\Big(I^{k}\big((\sigma'_s)^{k+1}-(\sigma'_s)^{k}\big)+(\sigma'_s)^{k+1}\overline{I}^{k+1}\Big)\text{d}\Omega' \text{d}v'\text{d}\Omega \text{d}v.
\end{split}
\end{equation*}

First, we estimate sequence $\overline{\rho}^{k+1}$. Multiplying $ (\ref{eq:1.2w})_1$ by $\overline{\rho}^{k+1}$ and integrating over $\mathbb{R}^3$, we have
\begin{equation}\label{go64}\begin{cases}
\frac{\text{d}}{\text{d}t}|\overline{\rho}^{k+1}|^2_2\leq A^k_\eta(t)|\overline{\rho}^{k+1}|^2_2+\eta |\nabla\overline{u}^k|^2_2,\\[8pt]
A^k_\eta(t)=C(|\nabla u^k|^2_{W^{1,q}}+\frac{1}{\eta}|\nabla\rho^{k}|^2_{3} +\frac{1}{\eta}|\rho^{k}|^2_{\infty}),\ \text{and} \ \int_0^t A^k_\eta(s)\text{d}s\leq \widehat{C}+\widehat{C}_{\eta}t
\end{cases}
\end{equation}
for $t\in[0,\overline{T}_1]$, where $0<\eta \leq \frac{1}{10}$ is a constant and $\widehat{C}_{\eta}$ is a positive constant  depending on $\frac{1}{\eta}$ and  constant $\widehat{C}$.

Second, we estimate sequence $\overline{I}^{k+1}$.
Multiplying $(\ref{eq:1.2w})_2$ by $\overline{I}^{k+1}$ and integrating over $\mathbb{R}^+\times S^2\times\mathbb{R}^3$, from (\ref{zhen1})-(\ref{jia345}) and the similar arguments used in Lemma \ref{lem:3} we have
\begin{equation}\label{go66}\begin{split}
&\frac{\text{d}}{\text{d}t}\|\overline{I}^{k+1}\|^2_{L^2(\mathbb{R}^+\times S^2;L^2(\mathbb{R}^3))}\\
\leq & \int_0^\infty \int_{S^2} \Big(|\sigma^{k+1}|_{\infty}|I^k|_{\infty}| \overline{\rho}^{k+1}|_2|\overline{I}^{k+1}|_2+ |\overline{\sigma}^{k+1,k}|_{\infty}| \rho^{k}|_\infty|I^k|_{\infty}| | \overline{\rho}^{k+1}|_2|\overline{I}^{k+1}|_2\Big)\text{d}\Omega \text{d}v\\
&+\int_0^\infty \int_{S^2} \int_0^\infty \int_{S^2}\frac{v}{v'}\overline{\sigma}_s\Big(|\rho|_\infty|\overline{I}'^k|_2|\overline{I}^{k+1}|_2+\|I'^k\|_{W^{1,q}}|\overline{I}^{k+1}|_2|\overline{\rho}^{k+1}|_2\Big) \text{d}\Omega' \text{d}v'\text{d}\Omega \text{d}v\\
&+\int_0^\infty \int_{S^2} \int_0^\infty \int_{S^2}\overline{\sigma}'_s\|I^k\|_{W^{1,q}}|\overline{I}^{k+1}|_2|\overline{\rho}^{k+1}|_2 \text{d}\Omega' \text{d}v'\text{d}\Omega \text{d}v\\
\leq& D^k_\eta(t)\|\overline{I}^{k+1}\|^2_{L^2(\mathbb{R}^+\times S^2;L^2(\mathbb{R}^3))}+|\overline{\rho}^{k+1}|^2_2+\eta\|\overline{I}^{k}\|^2_{L^2(\mathbb{R}^+\times S^2;L^2(\mathbb{R}^3))},
\end{split}
\end{equation}
where we have used the facts  $\sigma_a\geq 0$ and  $\sigma'_s\geq 0$. $\sigma^{k+1}$,
$\overline{\sigma}^{k+1,k}$
and $D^k_\eta(t)$ are  defined by
\begin{equation*}
\begin{split}
&\sigma^{k+1}=\sigma(v,\Omega,t,x,\rho^{k+1}),\quad \overline{\sigma}^{k+1,k}=\overline{\sigma}(v,\Omega,t,x,\rho^{k+1},\rho^k),\\
&D^k_\eta(t)=C\Big(\big(1+|\sigma^{k+1}|^2_\infty+|\overline{\sigma}^{k+1,k}|^2_\infty\big)\|I^{k}\|^2_{L^2(\mathbb{R}^+\times S^2; H^1\cap W^{1,q}(\mathbb{R}^3))}+\frac{1}{\eta}
|\rho|^2_{\infty}\Big).
\end{split}
\end{equation*}
From the  estimate  (\ref{pri}), we also have  $\int_0^t D^k_\eta(s)\text{d}s\leq \widehat{C}+\widehat{C}_{\eta}t$, for $t\in[0,\overline{T}_1]$.

Finally, multiplying $ (\ref{eq:1.2w})_3$ by $\overline{u}^{k+1}$ and integrating  over $\mathbb{R}^3$, we have
\begin{equation*}\begin{split}
&\frac{1}{2}\frac{\text{d}}{\text{d}t}|\sqrt{\rho}^{k+1}\overline{u}^{k+1}|^2_2+\mu|\nabla\overline{u}^{k+1} |^2_2+(\lambda+\mu)|\text{div}\overline{u}^{k+1} |^2_2\\
=& \int_{\mathbb{R}^3}\Big(-\overline{\rho}^{k+1}u^k_t\cdot\overline{u}^{k+1}-\overline{\rho}^{k+1}(u^{k-1}\cdot\nabla u^k)\cdot\overline{u}^{k+1}
-\rho^{k+1}(\overline{u}^{k}\cdot\nabla u^k)\cdot\overline{u}^{k+1}\Big)\text{d}x\\
&+\int_{\mathbb{R}^3}\Big(-\nabla \overline{p}^{k+1}_m\cdot\overline{u}^{k+1}+L_2\overline{u}^{k+1}\Big)\text{d}x
=:\sum_{i=5}^{14} I_i.
\end{split}
\end{equation*}

For the fluid terms $I_{5}-I_{8}$, according to the Gagliardo-Nirenberg inequality, Minkowski inequality and H\"older's inequality, it is not hard to show that
\begin{equation*}\begin{split}
& I_{5}= -\int_{\mathbb{R}^3}\overline{\rho}^{k+1}u^k_t\cdot\overline{u}^{k+1}\text{d}x \leq C\int_{\mathbb{R}^3}|\overline{\rho}^{k+1}||u^k_t||\overline{u}^{k+1}|\text{d}x, \\
& I_{6}= -\int_{\mathbb{R}^3}\overline{\rho}^{k+1}(u^{k-1}\cdot\nabla u^k)\cdot\overline{u}^{k+1}
\text{d}x\leq C|\overline{\rho}^{k+1}|_{2}|\nabla\overline{u}^{k+1}|_2\|\nabla u^k\|_1\|\nabla u^{k-1}\|_1,\\
&I_{7}=  -\int_{\mathbb{R}^3}\rho^{k+1}(\overline{u}^{k}\cdot\nabla u^k)\cdot\overline{u}^{k+1}\text{d}x\leq C|\overline{\rho}^{k+1}|^{\frac{1}{2}}_{\infty}|\sqrt{\rho}^{k+1}\overline{u}^{k+1}|_2\|\nabla u^k\|_1|\nabla \overline{u}^k|_2,\\
& I_{8}=-\int_{\mathbb{R}^3}\nabla \overline{p}^{k+1}_m\cdot\overline{u}^{k+1}\text{d}x=\int_{\mathbb{R}^3}
\overline{p}^{k+1}_m\text{div}\overline{u}^{k+1}\text{d}x
\leq C|\overline{p}^{k+1}_m|_{2}|\nabla\overline{u}^{k+1}|_2.
\end{split}
\end{equation*}
For the radiation related terms $I_{9}-I_{14}$,
\begin{equation*}\begin{split}
I_{9}=&-\frac{1}{c}\int_{\mathbb{R}^3}\int_0^\infty\int_{S^2}\Omega\cdot\overline{u}^{k+1}\Big(-\sigma^{k+1}_a\overline{I}^{k+1}\Big)\text{d}\Omega\text{d}v\text{d}x \quad\\
\leq & C|\sqrt{\rho}^{k+1}\overline{u}^{k+1}|_2|\rho^{k+1}|^{\frac{1}{2}}_\infty\|\overline{I}^{k+1}\|^2_{L^2(\mathbb{R}^+\times S^2;L^2(\mathbb{R}^3))}\|\sigma^{k+1}\|^2_{L^2(\mathbb{R}^+\times S^2;L^\infty(\mathbb{R}^3))},\\
 I_{10}=&-\frac{1}{c}\int_{\mathbb{R}^3}\int_0^\infty \int_{S^2}\Omega\cdot\overline{u}^{k+1}\Big(-I^k(\sigma^{k+1}_a-\sigma^{k}_a)\Big) \text{d}\Omega \text{d}v\text{d}x\\
\leq &C|\overline{\rho}^{k+1}|_{2}|\nabla\overline{u}^{k+1}|_2\|I^{k}\|_{L^2(\mathbb{R}^+\times S^2;H^1(\mathbb{R}^3))}\|\sigma^{k+1}\|_{L^2(\mathbb{R}^+\times S^2;L^\infty(\mathbb{R}^3))}\\
&+C|\rho^k|_\infty|\overline{\rho}^{k+1}|_2|\nabla\overline{u}^{k+1}|_2\|I^{k}\|_{L^2(\mathbb{R}^+\times S^2;H^1(\mathbb{R}^3))}\|\overline{\sigma}^{k+1,k}\|_{L^2(\mathbb{R}^+\times S^2;L^\infty(\mathbb{R}^3))},\\
 I_{11}=&-\frac{1}{c}\int_{\mathbb{R}^3}\int_0^\infty \int_{S^2}\int_0^\infty \int_{S^2}\frac{v}{v'}\Omega\cdot\overline{u}^{k+1}\sigma^{k+1}_s\overline{I}'^{k+1}\text{d}\Omega'\text{d}v'\text{d}\Omega \text{d}v\text{d}x \qquad\qquad\quad\\
 \leq &  C|\sqrt{\rho}^{k+1}\overline{u}^{k+1}|_2|\rho^{k+1}|^{\frac{1}{2}}_\infty\|\overline{I}^{k+1}\|_{L^2(\mathbb{R}^+\times S^2;L^2(\mathbb{R}^3))},\\
 \end{split}
\end{equation*}
\begin{equation*}\begin{split}
 I_{12}=&-\frac{1}{c}\int_{\mathbb{R}^3}\int_0^\infty \int_{S^2}\int_0^\infty \int_{S^2}\frac{v}{v'}\Omega\cdot\overline{u}^{k+1}I'^k(\sigma^{k+1}_s-\sigma^{k}_s)\text{d}\Omega'\text{d}v'\text{d}\Omega \text{d}v\text{d}x\\
 \leq &C|\nabla\overline{u}^{k+1}|_2|\overline{\rho}^{k+1}|_{2}\|I^k\|_{L^2(\mathbb{R}^+\times S^2;H^1(\mathbb{R}^3))},\\
 I_{13}=&\frac{1}{c}\int_{\mathbb{R}^3}\int_0^\infty \int_{S^2} \int_0^\infty \int_{S^2}\Omega \cdot\overline{u}^{k+1}I^{k}\Big((\sigma'_s)^{k+1}-(\sigma'_s)^{k}\Big)\text{d}\Omega' \text{d}v'\text{d}\Omega \text{d}v\text{d}x\\
 \leq& C|\nabla{u}^{k+1}|_2|\overline{\rho}^{k+1}|_{2}\|I^{k}\|_{L^2(\mathbb{R}^+\times S^2;H^1(\mathbb{R}^3))},\\
 I_{14}=&\frac{1}{c}\int_{\mathbb{R}^3}\int_0^\infty \int_{S^2} \int_0^\infty \int_{S^2}\Omega \cdot\overline{u}^{k+1} (\sigma'_s)^{k+1}\overline{I}^{k+1}\text{d}\Omega' \text{d}v'\text{d}\Omega \text{d}v\text{d}x\\
 \leq& C|\sqrt{\rho}^{k+1}\overline{u}^{k+1}|_2|\rho^{k+1}|^{\frac{1}{2}}_\infty\|\overline{I}^{k+1}\|_{L^2(\mathbb{R}^+\times S^2;L^2(\mathbb{R}^3))}.
\end{split}
\end{equation*}
Then combining the above estimates, we have
\begin{equation}\label{gogo1}\begin{split}
&\frac{\text{d}}{\text{d}t}|\sqrt{\rho}^{k+1}\overline{u}^{k+1}|^2_2+\mu|\nabla\overline{u}^{k+1} |^2\\
\leq& F^k_\eta(t)|\sqrt{\rho}^{k+1}\overline{u}^{k+1}|^2_2+F^k_2(t)|\overline{\rho}^{k+1}|^2_{2 }+M(C(c_0,c_1))|\overline{\rho}^{k+1}|^2_{2 }\\
&+F^k_3(t)\|\overline{I}^{k+1}\|^2_{L^2(\mathbb{R}^+\times S^2;L^2(\mathbb{R}^3))}+\eta|\nabla\overline{u}^{k}|^2_2+C\int_{\mathbb{R}^3}|\overline{\rho}^{k+1}||u^k_t||\overline{u}^{k+1}|\text{d}x,
\end{split}
\end{equation}
where
\begin{equation*}
\begin{split}
F^k_\eta(t)=&C\Big(1+\frac{1}{\eta}|\rho^{k+1}|_{\infty}\|\nabla u^{k}\|^2_1\Big),\\
F^k_2(t)=&C\big(\|\nabla u^{k-1}\|^2_1\|\nabla u^{k}\|^2_1+
\|I^{k}\|^2_{L^2(\mathbb{R}^+\times S^2;H^1(\mathbb{R}^3))}(\|\sigma^{k+1}\|^2_{L^2(\mathbb{R}^+\times S^2;L^\infty(\mathbb{R}^3))}+1)\big)\\
&+C|\rho^k|^2_\infty\|I^{k}\|^2_{L^2(\mathbb{R}^+\times S^2;H^1(\mathbb{R}^3))}\|\overline{\sigma}^{k+1,k}\|^2_{L^2(\mathbb{R}^+\times S^2;L^\infty(\mathbb{R}^3))},\\
F^k_3(t)=&C|\rho^{k+1}|_{\infty}\big(\|\sigma^{k+1}\|^2_{L^2(\mathbb{R}^+\times S^2;L^\infty(\mathbb{R}^3))}+1\big),\\
\end{split}
\end{equation*}
and we also have  $\int_0^t \big(F^k_\eta(s)+F^k_2(s)+F^k_3(s)\big)\text{d}s\leq \widehat{C}+\widehat{C}_\eta t$, for $t\in [0,\overline{T}_1]$.

To estimate $\int_{\mathbb{R}^3}|\overline{\rho}^{k+1}||u^k_t||\overline{u}^{k+1}|\text{d}x$,
we need the following Lemma.
\begin{lemma}[\textbf{The lower bound of the mass density at far field}]\label{zhengding}\  \\
There exists a sufficiently large $R>1$ and a time $\overline{T}_2\in (0,\overline{T}_1)$ small enough such that
$$
\frac{3}{8}\leq \rho^{k+1}(t,x)\leq \frac{5}{2}, \quad \forall \ (t,x)\in [0,\overline{T}_2]\times B^C_R,
$$
where the constants $R> 0$ and  $\overline{T}_2$ is independent of $k$, and $B^C_R=\mathbb{R}^3\setminus B_R$.
\end{lemma}
\begin{proof}
From $\rho_0-\overline{\rho}\in H^1\cap W^{1,q}$ and the embedding $W^{1,q}\hookrightarrow C_0$, where $C_0$ is the set of all continuous functions on
$\mathbb{R}^3$ vanishing at infinity, we can choose a sufficiently large $R> 1$  such that
\begin{equation}\label{zhenglem}
\frac{3}{8}\overline{\rho}\leq \rho_0\leq \frac{5}{2}\overline{\rho} \quad \text{for} \quad x \in B^C_R.
\end{equation}
From the proof of Lemma \ref{lem1} we know that
\begin{equation}\label{zhengleme}
\rho^{k+1}(t,x)=\rho_0(U^{k+1}(0;t,x))\exp\Big(-\int_{0}^{t}\textrm{div} u^k(s;U^{k+1}(s,t,x))\text{d}s\Big),
\end{equation}
where  $U^{k+1}\in C([0,\overline{T}_1]\times[0,\overline{T}_1]\times \mathbb{R}^3)$ is the solution to the initial value problem
\begin{equation}\label{gongshi}
\begin{cases}
\frac{d}{ds}U^{k+1}(s;t,x)=u^k(s,U^{k+1}(s;t,x)),\quad 0\leq s\leq \overline{T}_1,\\[8pt]
U^{k+1}(t,t,x)=x, \qquad \qquad \qquad 0\leq t\leq \overline{T}_1,\ x\in \mathbb{R}^3.
\end{cases}
\end{equation}
The local estimate (\ref{pri}) leads to
\begin{equation}\label{zhengle}
\begin{split}
&\int_0^t|\textrm{div} u^k(s,U^{k+1}(s,t,x))|\text{d}s\leq \int_0^t|\nabla u^k|_\infty\text{d}s\leq \overline{C}t^{1/2}\leq \ln 2.
\end{split}
\end{equation}
From the ODE problem (\ref{gongshi}), we get
\begin{equation}\label{zhengle1}
\begin{split}
&|U^{k+1}(0;t,x)-x|=|U^{k+1}(0;t,x)-U^{k+1}(t;t,x)|\\
\leq& \int_0^t| u^k(\tau,U^{k+1}(\tau;t,x))|\text{d}\tau\leq \overline{C}t\leq R/2,
\end{split}
\end{equation}
for all $(t,x)\in [0,\overline{T}_2]\times \mathbb{R}^3$, where $\overline{C}$ is a positive constant independent of $k$, and $\overline{T}_2$ is a small positive time depending only on $\overline{C}$ and $\overline{T}_1$. That means,
 $$ U^{k+1}(0,t,x)\in B^C_{R/2}, \quad \text{for}\quad  (t,x)\in [0,\overline{T}_2]\times B^C_{R}.$$
 Then combining (\ref{zhenglem}), (\ref{zhengleme}), (\ref{zhengle}) and (\ref{zhengle1}),  the desired conclusion is obtained.
\end{proof}
Based on Lemma \ref{zhengding}, for $t\in [0,\overline{T}_2]$, we have
\begin{equation}\label{zhengchu}
\begin{split}
&\int_{B_R}|\overline{\rho}^{k+1}||u^k_t||\overline{u}^{k+1}|\text{d}x
\leq C|\nabla u^k_t |^2_2|\overline{\rho}^{k+1} |^2_2+\frac{1}{8}\mu|\nabla\overline{u}^{k+1} |^2_2,
\end{split}
\end{equation}
and
\begin{equation}\label{zhengchu1}
\begin{split}
&\int_{B^C_R}|\overline{\rho}^{k+1}||u^k_t||\overline{u}^{k+1}|\text{d}x
\leq\frac{C}{\sqrt{\rho}^{k+1}} |\overline{\rho}^{k+1} |_2|\nabla u^k_t |_2|\sqrt{\rho}^{k+1}\overline{u}^{k+1}|_3\\
&\quad \quad   \leq C|\nabla u^k_t |^2_2|\overline{\rho}^{k+1} |^2_2+\frac{1}{8}\mu|\nabla\overline{u}^{k+1} |^2_2+C|\sqrt{\rho}^{k+1}\overline{u}^{k+1}|^2_2.
\end{split}
\end{equation}
Define
\begin{equation*}\begin{split}
\Gamma^{k+1}=&\sup_{0\leq t \leq \overline{T}_2}\Big(\|\overline{I}^{k+1}(t)\|^2_{L^2(\mathbb{R}^+\times S^2;L^2(\mathbb{R}^3))}+|\overline{\rho}^{k+1}(t)|^2_{
 2}+|\sqrt{\rho}^{k+1}\overline{u}^{k+1}(t)|^2_2\Big),
\end{split}
\end{equation*}
then from (\ref{go64})-(\ref{gogo1}), (\ref{zhengchu}) and  (\ref{zhengchu1}) we have
\begin{equation*}\begin{split}
&\Gamma^{k+1}(\overline{T}_2)+\mu\int_0^{\overline{T}_2} |\nabla\overline{u}^{k+1} |^2_2\text{d}t\\
\leq& \int_0^{\overline{T}_2} G^k_{\eta} \Gamma^{k+1}(t)\text{d}t+3\eta\int_0^{\overline{T}_2}|\nabla\overline{u}^{k}(t) |^2_2\text{d}t+ \eta \overline{T}_2 \sup_{0\leq t \leq \overline{T}_2}\|\overline{I}^{k}(t)\|^2_{L^2(\mathbb{R}^+\times S^2;L^2(\mathbb{R}^3))}
\end{split}
\end{equation*}
for some $G^k_{\eta}$ such that  $\int_{0}^{t}G^k_\eta(s)\text{d}s\leq \widehat{C}+\widehat{C}_{\eta} t$ for $0\leq t \leq \overline{T}_2$.
By using Gronwall's inequality, we have
\begin{equation*}\begin{split}
&\Gamma^{k+1}(\overline{T}_2)+ \mu\int_{0}^{\overline{T}_2}|\nabla\overline{u}^{k+1}|^2_2\text{d}s\\
\leq & \exp{(\widehat{C}+\widehat{C}_{\eta}t)}\Big(3\eta\int_0^{\overline{T}_2}|\nabla\overline{u}^{k}(t) |^2_2\text{d}t+ \eta \overline{T}_2 \sup_{0\leq t \leq \overline{T}_2}\|\overline{I}^{k}(t)\|^2_{L^2(\mathbb{R}^+\times S^2;L^2(\mathbb{R}^3))}\Big).
\end{split}
\end{equation*}
Since $0< \overline{T}_2\leq 1$, we first choose $\eta=\eta_0$ small enough such that
$$
3\eta_0\exp(\widehat{C})\leq \min\Big\{\frac{1}{8},\ \frac{\mu}{8}\Big\},
$$
then we choose $\overline{T}_2=T_*$ small enough such that
$$ 3\eta_0\exp(\widehat{C}_{\eta_0} T_*)\leq 4.$$
So, when $\Gamma^{k+1}=\Gamma^{k+1}(T_*)$, we have
\begin{equation*}\begin{split}
\sum_{k=1}^{\infty}\Big( \Gamma^{k+1}(T_*)+\mu\int_{0}^{T_*} |\nabla\overline{u}^{k+1}|^2_2\text{d}s\Big)\leq \widehat{C}<+\infty.
\end{split}
\end{equation*}
Therefore, the Cauchy sequence $(I^k,\rho^k,u^k)$ converges to a limit $(I,\rho,u)$ in the following strong sense:
\begin{equation}\label{strong}
\begin{split}
&I^k\rightarrow I \ \text{in}\ L^\infty([0,T_*];L^2(R^+\times S^2;L^2(\mathbb{R}^3))),\\
&\rho^k\rightarrow \rho \ \text{in}\ L^\infty([0,T_*];L^2(\mathbb{R}^3)),\\
&u^k\rightarrow u\ \text{in}\ L^2([0,T_*];\mathbb{D}^1(\mathbb{R}^3)).
\end{split}
\end{equation}
Thanks to the local uniform  estimate (\ref{pri}), the strong convergence in (\ref{strong}) and the lower semi-continuity of norms, we also have $(I,\rho, u)$ still satisfies the a priori estimates (\ref{pri}).
Then it is easy to show that $(I,\rho,u)$ is a weak solution in the sense of distribution satisfying the a priori estimates (\ref{pri}).\\[2mm]
\underline{Step 2}. The uniqueness  of strong solutions. Let $(I_1,\rho_1,u_1)$ and $(I_2,\rho_2,u_2)$ be two strong solutions to Cauchy problem (\ref{eq:1.2})-(\ref{eq:2.2hh}) satisfying the a priori estimates (\ref{pri}). Denote
$$
\overline{I}=I_1-I_2,\quad  \overline{\rho}=\rho_1-\rho_2,\quad \overline{p}_m=p_m(\rho_1)-p_m(\rho_2),\quad  \overline{u}=u_1-u_2.
$$
Let
$$
\Gamma(t)=\|\overline{I}(t)\|^2_{L^2(\mathbb{R}^+\times S^2;L^2(\mathbb{R}^3))}+|\overline{\rho}(t)|^2_{ 2}+|\sqrt{\rho}_1\overline{u}|^2_2.
$$
By the same method for deriving (\ref{go64})-(\ref{gogo1}), we similarly have
\begin{equation}\label{gonm}\begin{split}
\frac{\text{d}}{\text{d}t}\Gamma(t)+\mu|\nabla \overline{u}|^2_2\leq H(t)\Gamma (t),
\end{split}
\end{equation}
where $ \int_{0}^{t}H(s)ds\leq \widehat{C}$, for $t\in [0,T_*]$.  Then from the Gronwall's inequality, we immediately conclude that
$$\overline{I}=\overline{\rho}=\nabla \overline{u}=0,$$
Since $\overline{u}(t,x)\rightarrow 0$ as $|x|\rightarrow \infty$, the uniqueness follows.\\[2mm]
\underline{Step 3}. The time-continuity of the strong solution. It can be obtained by the same method used in the proof of Lemma \ref{lem1}.

\end{proof}

\subsubsection{\bf{ Case $\overline{\rho}=0$.}}\  \\

\begin{proof}Multiplying the first equation in (\ref{eq:1.2w}) by $\text{sign}(\overline{\rho}^{k+1})|\overline{\rho}^{k+1}|^{\frac{1}{2}}$ and integrating over $\mathbb{R}^3$, we have
\begin{equation}\label{go65e}\begin{cases}
\displaystyle \frac{\text{d}}{\text{d}t}|\overline{\rho}^{k+1}|^2_{3/2}\leq B^k_\eta(t)|\overline{\rho}^{k+1}|^2_{3/2}+\eta |\nabla\overline{u}^k|^2_2,\\[8pt]
\displaystyle
B^k_\eta(t)=C\left(|\nabla u^k|_{W^{1,q}}+\frac{1}{\eta}|\rho^{k}|^2_{H^1}\right),\ \text{and} \ \int_0^t B^k_\eta(s)\text{d}s\leq \widehat{C}+\widehat{C}_{\eta}t,
\end{cases}
\end{equation}
 for $t\in [0,\overline{T}_1]$.
Similarly, we can establish the estimates
\begin{equation}\label{go64e}\begin{cases}
\displaystyle \frac{\text{d}}{\text{d}t}|\overline{\rho}^{k+1}|^2_2\leq A^k_\eta(t)|\overline{\rho}^{k+1}|^2_2+\eta |\nabla\overline{u}^k|^2_2,\\[8pt]
\displaystyle \frac{\text{d}}{\text{d}t}\|\overline{I}^{k+1}\|^2_{L^2(\mathbb{R}^+\times S^2;L^2(\mathbb{R}^3))}\\[8pt]
\leq D^k_\eta(t)\|\overline{I}^{k+1}\|^2_{L^2(\mathbb{R}^+\times S^2;L^2(\mathbb{R}^3))}+|\overline{\rho}^{k+1}|^2_2+\eta\|\overline{I}^{k}\|^2_{L^2(\mathbb{R}^+\times S^2;L^2(\mathbb{R}^3))},\\[8pt]
\displaystyle \frac{1}{2}\frac{\text{d}}{\text{d}t}|\sqrt{\rho}^{k+1}\overline{u}^{k+1}|^2_2+\mu\|\nabla\overline{u}^{k+1} \|^2\\[8pt]
\leq F^k_\eta(t)|\sqrt{\rho}^{k+1}\overline{u}^{k+1}|^2_2+F^k_2(t)|\overline{\rho}^{k+1}|_{L^2\cap L^{\frac{3}{2}}}+M(C(c_0,c_1))|\overline{p}^{k+1}_m|^2_2\\[8pt]
\displaystyle
+F^k_3(t)\|\overline{I}^{k+1}\|^2_{L^2(\mathbb{R}^+\times S^2;L^2(\mathbb{R}^3))}
+\eta|\nabla\overline{u}^{k}|^2_2+C\int_{\mathbb{R}^3}|\overline{\rho}^{k+1}||u^k_t||\overline{u}^{k+1}|\text{d}x,
\end{cases}
\end{equation}
for $t\in [0,\overline{T}_1]$, and from the local a priori estimate (\ref{pri}) we have
\begin{equation*}\begin{split}
\int_0^t \big(A^k_\eta(s)+D^k_\eta(s)+F^k_\eta(s)+F^k_2(s)+F^k_3(s)\big)\text{d}s\leq \widehat{C}+\widehat{C}_\eta t.
\end{split}
\end{equation*}
According to (\ref{go65e}), the key term can be estimated by
\begin{equation}\label{go6ge}\begin{split}
&\int_{\mathbb{R}^3}|\overline{\rho}^{k+1}||u^k_t||\overline{u}^{k+1}|\text{d}x
\leq C|\nabla u^k_t |^2_2|\overline{\rho}^{k+1} |^2_{3/2}+\frac{1}{8}\mu|\overline{u}^{k+1}|^2_2.
\end{split}
\end{equation}
We can also define the energy function by
\begin{equation*}\begin{split}
\Gamma^{k+1}=&\sup_{0\leq t \leq \overline{T}_1}\Big(\|\overline{I}^{k+1}(t)\|^2_{L^2(\mathbb{R}^+\times S^2;L^2(\mathbb{R}^3))}+|\overline{\rho}^{k+1}(t)|^2_{
 2}+|\sqrt{\rho}^{k+1}\overline{u}^{k+1}(t)|^2_2\Big).
\end{split}
\end{equation*}
Then from (\ref{go65e})-(\ref{go6ge}) and Gronwall's inequality, we  easily obtain
\begin{equation*}\begin{split}
&\Gamma^{k+1}(\overline{T}_1)+\mu\int_0^{\overline{T}_1} |\nabla\overline{u}^{k+1} |^2_2\text{d}t\\
\leq& \int_0^{\overline{T}_1} G^k_{\eta} \Gamma^{k+1}(t)\text{d}t+3\eta\int_0^{\overline{T}_1}|\nabla\overline{u}^{k}(t) |^2_2\text{d}t+ \eta \overline{T}_1 \sup_{0\leq t \leq \overline{T}_1}\|\overline{I}^{k}(t)\|^2_{L^2(\mathbb{R}^+\times S^2;L^2(\mathbb{R}^3))}
\end{split}
\end{equation*}
for some $G^k_{\eta}$ such that  $\int_{0}^{t}G^k_\eta(s)\text{d}s\leq \widehat{C}+\widehat{C}_{\eta} t$ for $0\leq t \leq \overline{T}_1$.

The rest of the proof are analogous to the proof for the case $\overline{\rho}>0$. We omit the details here.
\end{proof}
Thus the proof of Theorem \ref{th1} is finished.

\section{Necessity and sufficiency of the initial layer compatibility condition}

We prove Theorem \ref{th2} in this section, that is, the initial layer compatibility condition is not only sufficient but also necessary if the initial vacuum set is not very irregular. Since the strong solution $(I,\rho,u)$ only satisfies the Cauchy problem in the sense of distribution, we only have
$$I(v,\Omega,0,x)=I_0,\ \rho(0,x)=\rho_0,\ \rho u(0,x)=\rho_0 u_0,\ x\in \mathbb{R}^3.$$
 So, the key point of the proof is to make sure that  the relation $u(0,x)=u_0$  holds in the vacuum domain. Now we give the proof of Theorem \ref{th2}.

\begin{proof} We prove the necessity and sufficiency, respectively.\\
\underline{\bf{Step 1: to prove the necessity.}} Let $(I,\rho,u)$ be a strong solution of the Cauchy problem  (\ref{eq:1.2})-(\ref{eq:2.2hh}) with the regularity as shown in Definition \ref{strong1}. Then from the momentum equations in  (\ref{eq:1.2}) we have
\begin{equation}\label{da9}\begin{split}
Lu(t)+\nabla p_{m}(t)+\frac{1}{c}\int_0^\infty \int_{S^2}A_r(t)\Omega \text{d}\Omega \text{d}v=\sqrt{\rho}(t) G(t)
\end{split}
\end{equation}
for  $0\leq t \leq T_*$, where $G(t)=\rho^{\frac{1}{2}}(t) (-u_t-u\cdot\nabla u)$. Since
$$
\sqrt{\rho}u_t\in  L^\infty([0,T_*];L^2),\quad  \sqrt{\rho}u\cdot\nabla u \in  L^\infty([0,T_*];L^2),
$$
we have $G(t)\in  L^\infty([0,T_*];L^2)$. So there exists a sequence $\{t_k\}$, $t_k\rightarrow 0$, such that
$$
G(t_k)\rightarrow g \quad \text{in} \quad L^2 \quad \text{for some} \quad g\in L^2.
$$
Taking $t=t_k\rightarrow 0$ in (\ref{da9}), we obtain
\begin{equation}\label{jojo}\begin{split}
Lu(0)+\nabla p_{m}(\rho(0))+\frac{1}{c}\int_0^\infty \int_{S^2}A_r(0)\Omega \text{d}\Omega \text{d}v=\sqrt{\rho}(0) g.
\end{split}
\end{equation}
Together with the strong convergence (\ref{coco}) and the construction of our strong solutions, initial layer compatibility condition \eqref{kkkkk} holds with $g_1=g$.\\
\underline{\bf{Step 2: to prove the sufficiency.}} Let $(I_0,\rho_0,u_0)$ be the initial data satisfying (\ref{gogo})-(\ref{kkkkk}). Then there exists a unique strong solution $(I,\rho,u)$ to the Cauchy problem  (\ref{eq:1.2})-(\ref{eq:2.2hh}) with the regularity
\begin{equation*}
\begin{split}
&I(v,\Omega,t,x)\in L^2(\mathbb{R}^+\times S^2;C([0,T_*];H^1\cap W^{1,q}(\mathbb{R}^3))), \\
&\rho(t,x)-\overline{\rho}\in C([0,T_*];H^1\cap W^{1,q}), \ u(t,x)\in C([0,T_*];\mathbb{D}^1\cap D^2(\mathbb{R}^3)).
\end{split}
\end{equation*}
So we only need to verify the initial conditions
$$
I(v,\Omega,0,x)=I_0,\ \rho(0,x)=\rho_0,\ u(0,x)=u_0(x), \ x\in \mathbb{R}^3.
$$
From the weak formulation of the strong solution, it is easy to know that
$$
I(v,\Omega,0,x)=I_0,\ \rho(0,x)=\rho_0,\ \rho(0,x)u(0,x)=\rho_0u_0, \ x\in \mathbb{R}^3.
$$
So it remains to prove that $u(0,x)=u_0(x), \ x\in V$. Let $\overline{u}_0(x)=u_0(x)-u(0,x)$. According to the proof of the necessity, we know that $(I(v,\Omega,0,x),\rho(0,x),u(0,x))$ also satisfies the relation (\ref{kkkkk}) for $g_1\in L^2$. Then $\overline{u}_0\in \mathbb{D}^1_0(V)\cap D^2(V)$ is the unique solution of the elliptic problem (\ref{zhen101}) in $V$ and thus $\overline{u}_0=0$ in $V$, which implies that
$u(0,x)=u_0(x), \ x\in V$.
\end{proof}

Finally we remark that, for a special case that the mass density $\rho(t,x)= 0$  only holds in some single point or only decay in the far field, $u(0,x)=u_0$ obviously hold according to our proof of the sufficiency.

\section{Blow-up criterion of strong solutions}

In this section, we prove Theorem \ref{th3} in which we establish a blow-up criterion for strong solutions. Firstly we define the following two auxiliary quantities:
\begin{equation*}\begin{split}
\Phi(t)=&1+\|I\|_{L^2(\mathbb{R}^+\times S^2; C([0,t];H^1\cap W^{1,q}(\mathbb{R}^3)))}+\sup_{0\leq s\leq t}\|\rho(s)-\overline{\rho}\|_{H^1\cap W^{1,q}}+\sup_{0\leq s\leq t}|u(s)|_{\mathbb{D}^1},\\[6pt]
\Theta(t)=&1+\|I\|_{L^2(\mathbb{R}^+\times S^2; C([0,t];H^1\cap W^{1,q}(\mathbb{R}^3)))}+\|I_t\|_{L^2(\mathbb{R}^+\times S^2; C([0,t];L^2\cap L^q(\mathbb{R}^3)))}\\[6pt]
&+\|\rho(t)-\overline{\rho}\|_{H^1\cap W^{1,q}}+|\rho_t(t)|_{L^2\cap L^q}
\\
&+|u(t)|_{\mathbb{D}^1\cap D^2}+|\sqrt{\rho}u_t(t)|_2+\int_{0}^{t}\Big(|u(s)|^2_{D^{2,q}}+|u_t(s)|^2_{\mathbb{D}^1}\Big)\text{d}s.
\end{split}
\end{equation*}
Then according to the definition of the maximal existence time $\overline{T}$ of the local strong solution in Definition (\ref{strong1}), we know that
\begin{equation}\label{up}\begin{split}
\Theta(t)\rightarrow +\infty, \quad \text{as} \quad t\rightarrow \overline{T}.
\end{split}
\end{equation}
Based on (\ref{up}), our purpose in the following proof is to show that
\begin{equation}\label{ga1}\begin{split}
\Phi(t)\rightarrow+\infty, \quad \text{as} \quad t\rightarrow \overline{T}.
\end{split}
\end{equation}

\textbf{Now we give the proof for  Theorem {\ref{th3}}.}
\begin{proof}
Firstly, let $(I,\rho,u)$ be the unique strong solution of the Cauchy problem  (\ref{eq:1.2})-(\ref{eq:2.2hh}) with the regularities shown in Definition \ref{strong1}. Then from similar arguments as shown in Lemmas \ref{lem:2} and \ref{lem:3}, for $0< t < \overline{T}$, we easily get
\begin{equation}\label{vb1}
\qquad \qquad \|\rho(t)-\overline{\rho}\|_{H^1\cap W^{1,q}}\leq C\Big(1+\int_0^t \|\nabla u\|_{H^1\cap W^{1,q}}\text{d}s\Big) \exp\Big(\int_0^t \|\nabla u\|_{W^{1,q}}\text{d}s\Big),
\end{equation}
and
\begin{equation}\label{vb1qq}
\quad \|I\|^2_{L^2(\mathbb{R}^+\times S^2; C([0,t]; H^1\cap W^{1,q}(\mathbb{R}^3)))}
\leq  \exp\big(t M(\Phi(t)) \big)\big(1+tM(\Phi(t))\big).
\end{equation}
Directly from the continuity equation and  the radiation transfer equation in (\ref{eq:1.2}),  for $0< t < \overline{T}$, we deduce that
\begin{equation}\label{vb2}\begin{cases}
|(\rho_t,(p_m)_t)(t)|_{2}\leq M(|\rho(t)|_\infty)\big(1+|\rho(t)|_\infty |\nabla u(t)|_{2} +|u(t)|_6 |\nabla\rho(t)|_{3}\big),\\[6pt]
|(\rho_t,(p_m)_t)(t)|_{q}\leq M(|\rho(t)|_\infty)\big(1+|\rho(t)|_\infty |\nabla u(t)|_{q} +|u(t)|_\infty |\nabla\rho(t)|_q\big),\\[6pt]
%
%
\|I_t(t)\|_{L^2(\mathbb{R}^+\times S^2; C([0,t];L^2\cap L^q(\mathbb{R}^3)))}\leq M(\Phi(t)).
\end{cases}
\end{equation}

Secondly, for the velocity vector $u$ of the fluid, standard energy estimates as shown in Lemma \ref{lem:4} lead to
\begin{equation}\label{vb4}\begin{split}
&\frac{1}{2}\frac{\text{d}}{\text{d}t}\int_{\mathbb{R}^3}\rho |u_t|^2 \text{d}x+\int_{\mathbb{R}^3}\big(\mu|\nabla u_t|^2+(\lambda+\mu)|\text{div}u_t|^2\big)\text{d}x\\
\leq& C\big(1+|\rho|^3_{\infty}|\nabla u|^4_2 \big)|\sqrt{\rho}u_t|^2_{2}+C\big(1+|\rho_t|^2_{2}|\nabla u|^4_2 +|(p_m)_t|^2_2\big)\\
&+ M(|\rho|_\infty)\big(1+|\rho|_\infty)\big(1+|\rho_t|^2_2\|I(t)\|^2_{L^2(\mathbb{R}^+\times S^2; L^2(\mathbb{R}^3))}+\|I_t(t)\|^2_{L^2(\mathbb{R}^+\times S^2; L^2(\mathbb{R}^3))}\big).
\end{split}
\end{equation}
Integrating (\ref{vb4}) over $(\tau,t)$ for $\tau\in (0,t)$, and according to (\ref{vb1})-(\ref{vb2}),  we  easily have
\begin{equation}\label{vb7}
\begin{split}
|\sqrt{\rho}u_t(t)|^2_{2}+\int_{\tau}^{t}|\nabla u_t|^2_{2}\text{d}s
\leq&C+|\sqrt{\rho}u_t(\tau)|^2_{2}
+C\int_{\tau}^{t}(1+|\sqrt{\rho}u_t|^2_{2})M(\Phi)\text{d}s.
\end{split}
\end{equation}
According to Gronwall's inequality, we have
\begin{equation}\label{vb11}
\begin{split}
|\sqrt{\rho}u_t(t)|^2_{2}+\int_{\tau}^{t}|\nabla u_t(s)|^2_{2}\text{d}s
\leq&C\Theta (\tau)\exp\big(\overline{T}M(\Phi(t))\big).
\end{split}
\end{equation}

Thirdly, we consider the higher order terms  $|u|_{D^2}$ and $|\nabla u|_{W^{1,q}}$.
From the standard  elliptic regularity estimates and Minkowski inequality, we have
\begin{equation*}
\begin{split}
| u|_{D^2}\leq & C\Big(|\rho u_t|_2+|\rho u\cdot\nabla u|_2+|\nabla p_m|_2+\int_0^\infty \int_{S^2}|A_r|_2 \text{d}\Omega \text{d}v\Big)\\
\leq & C\big(1+|\rho|^{\frac{1}{2}}_\infty |\sqrt{\rho}u_t|_{2}+|\rho|_\infty | u|^{\frac{3}{2}}_{\mathbb{D}^1}| u|^{\frac{1}{2}}_{D^2}+M(|\rho|_\infty)(|\rho|_2+\|I\|_{L^2(\mathbb{R}^+\times S^2; L^2(\mathbb{R}^3))})\big).
\end{split}
\end{equation*}
This implies, by Young's inequality,  that,
\begin{equation}\label{vb5}
\begin{split}
|  u(t)|_{\mathbb{D}^1\cap D^2}\leq & C\big(1+|\sqrt{\rho} u_t(t)|_2\big)M(\Phi(t))
\end{split}
\end{equation}
for $t\in (\tau,\overline{T})$.
Similarly, due to (\ref{vb5}), for $t\in (\tau,\overline{T})$, we have
\begin{equation}\label{vb6}
\begin{split}
 |u(t)|_{D^{2,q}}
\leq C\big( (1+|\sqrt{\rho} u_t(t)|^2_2)M(\Phi(t))+|\nabla u_t(t)|_2\big).
\end{split}
\end{equation}

Finally,  we combine (\ref{vb1})-(\ref{vb6}) and conclude that for each $t\in (\tau,\overline{T})$,
\begin{equation}\label{vb12}
\begin{split}
\Theta(t)
\leq&C(1+\overline{T})(1+\Theta (\tau))^2 M(\Phi(t))\exp\big(\overline{T}M(\Phi(t))\big).
\end{split}
\end{equation}
From (\ref{up}), the blow-up criterion \eqref{blowcr} as shown in Theorem \ref{th3} follows immediately by letting $t\rightarrow \overline{T}$ in (\ref{vb12}).

\end{proof}

\bigskip

{\bf Acknowledgement:} The research of  Y. Li and S. Zhu were supported in part
by National Natural Science Foundation of China under grant 11231006 and Natural Science Foundation of Shanghai under grant 14ZR1423100. S. Zhu was also supported by China Scholarship Council.

\end{document}